\documentclass[10pt]{amsart}

\usepackage{graphicx}
\usepackage{amsmath}
\usepackage{amssymb}
\usepackage{mathrsfs}
\usepackage{hyperref}
\usepackage{esvect}
\newtheorem{theorem}{Theorem}
\newtheorem{lemma}[theorem]{Lemma}
\newtheorem{proposition}{Proposition}

\newtheorem{definition}{Definition}
\newtheorem{example}{Example}
\newtheorem{remark}{Remark}
\newtheorem{ques}{Question}

\numberwithin{equation}{section}

\renewcommand{\binom}[2]{\left(\genfrac{}{}{0pt}{}{#1}{#2}\right)}

\begin{document}
\title[Solutions to $ {\rm SU}(n+1) $ Toda system via toric curves]{Solutions to ${\rm SU}(n+1)$ Toda system with cone singularities via toric curves on compact Riemann surfaces}
	
	%    Information for first author
	\author{Jingyu Mu}
	%    Address of record for the research reported here
	\address{School of Mathematical  Sciences, University of Science and Technology of China, Hefei 230026 China
	}
	\email{jingyu@mail.ustc.edu.cn}
	%    \thanks will become a 1st page footnote.
	\thanks{
		Y.S. is supported in part by the National Natural Science Foundation of China
		(Grant No. 11931009) and Anhui Initiative in Quantum Information Technologies (Grant No. AHY150200).
		B.X. is supported in part by the Project of Stable Support for Youth Team in Basic Research Field, CAS (Grant No. YSBR-001) and the National Natural Science Foundation of China (Grant Nos. 12271495, 11971450 and 12071449).  %Both Y.S. and B.X. are supported in part by
		%the Fundamental Research Funds for the Central Universities (Grant No. WK3470000015).
	}
	\thanks{$^\dagger$B.X. is the corresponding author.}
	
	%    Information for second author
	\author{Yiqian Shi}
	\address{School of Mathematical Sciences and CAS Wu Wen-Tsun Key Laboratory of Mathematics, University of Science and Technology of China, Hefei 230026 China}
	\email{yqshi@ustc.edu.cn}
	\thanks{}
		
	\author{Bin Xu$^\dagger$}
	\address{School of Mathematical Sciences and CAS Wu Wen-Tsun Key Laboratory of Mathematics, University of Science and Technology of China, Hefei 230026 China}
	\email{bxu@ustc.edu.cn}
	%    General info
	
	\date{}
	
	\dedicatory{}
	

%\begin{abstract}
%In this study, we examine the construction of solutions for the SU(n+1) Toda system with cone singularities on compact Riemann surfaces, utilizing %{\it toric curves} and {\it character ensembles}. 

 %Our findings contribute to the broader understanding of the SU(n+1) Toda system's solution space on compact Riemann surfaces, expanding the scope %of mathematical frameworks applicable to differential geometry and mathematical physics.

 \begin{abstract}
On a compact Riemann surface \(X\) with finite punctures \(P_1, \ldots, P_k\), we define \textit{toric curves} as multi-valued, totally unramified holomorphic maps to \(\mathbb{P}^n\) with monodromy in a maximal torus of \({\rm PSU}(n+1)\). \textit{Toric solutions} for the \({\rm SU}(n+1)\) system on $X\setminus\{P_1,\ldots, P_k\}$ are recognized by their associated {\it toric} curves in \(\mathbb{P}^n\). We introduce a \textit{character n-ensemble} as an \(n\)-tuple of meromorphic one-forms with simple poles and purely imaginary periods, generating toric curves on \(X\) minus finitely many points. We establish on $X$ a correspondence between character $n$-ensembles and toric solutions to the \({\rm SU}(n+1)\) system  with finitely many cone singularities. Our approach not only broadens seminal solutions for up to two cone singularities on the Riemann sphere, as classified by Jost-Wang (\textit{Int. Math. Res. Not.}, (6):277-290, 2002) and Lin-Wei-Ye (\textit{Invent. Math.}, 190(1):169-207, 2012), but also advances beyond the limits of Lin-Yang-Zhong's existence theorems (\textit{J. Differential Geom.}, 114(2):337-391, 2020) by introducing a new solution class.
\end{abstract}

\maketitle

\section {Introduction}

Gervais-Matsuo \cite[Section 2.2.]{GM:1993} were the first to demonstrate that totally unramified holomorphic curves \cite[p.\,270, Proposition]{GH:1994} in \({\mathbb{P}}^n\) yield local solutions to \({\rm SU}(n+1)\) Toda system, interpreting these systems as the infinitesimal Pl\"ucker formula \cite[p.\,269]{GH:1994} for these curves. Adam Doliwa \cite{Do:1997} expanded upon this foundation, extending their findings to Toda systems associated with non-exceptional simple Lie algebras.  Over the ensuing years, Luca Battaglia, Chang-shou Lin, and their collaborators \cite{Batt_JMAA2015, Batt_Adv2015, Batt2016, KLNW2022, LYZ2020} have devoted significant research efforts to classifying and empirically substantiating existence theorems for solutions to Toda systems. These systems are associated with any complex simple Lie algebra and feature cone singularities on compact Riemann surfaces (Definition \ref{def:sol_cone}).
In the next three subsections, we aim to highlight the principal advancements that are pertinent to our thematic focus within this field. To lay the foundation for our discussion, we initially explore the basic correspondence between solutions to the \(\mathrm{SU}(n+1)\) Toda system with cone singularities and their associated unitary curves in \(\mathbb{P}^n\) with regular singularities (Definition \ref{def:curve_reg_sing}) in the first two subsections.

\subsection{A differential-geometric framework for solutions to Toda systems}

Building on the seminal contributions of Givental \cite{Giv1989} and Positselski \cite{Pos1991}, this manuscript establishes on Riemann surfaces a new conceptual framework for solutions to \(\mathrm{SU}(n+1)\) Toda system with cone singularities (Definition \ref{def:sol_cone}).
 As elucidated in 
\cite[Theorem 1.2. (ii)]{MSSX2024}, our analytical approach scrutinizes the solutions as \(n\)-tuples of Kähler metrics with cone singularities on Riemann surfaces. Importantly, our methodology not only aligns with but also translates into the ones employed by a spectrum of researchers including Jost-Wang \cite{JW:2002}, Lin-Wei-Ye \cite{LWY:2012}, Battaglia \cite{Batt_JMAA2015, Batt_Adv2015, Batt2016}, Lin-Yang-Zhong \cite{LYZ2020}, Chen-Lin \cite{CL2022, CL2023}, and Chen \cite{Chen2023}, thereby establishing a conceptual equivalence across these varied approaches. This framework is versatile, allowing for an extension to Toda systems associated with any complex simple Lie algebra on K\"ahler manifolds.

Consider a Riemann surface $\mathfrak{X}$, not necessarily compact, alongside a closed discrete subset $\mathfrak{S}$ within $\mathfrak{X}$, where notably, $\mathfrak{S}$ is at most countable and possibly empty. We consider a vector $\vv{\omega}$ of Kähler metrics on the punctured Riemann surface $\mathfrak{X}\setminus\mathfrak{S}$,  defined as
\begin{equation}
\label{equ:metric_vector}    
\vv{\omega}:=\left(\omega_1=\frac{{\rm i}}{2}e^{u_1}\mathrm{d}z\wedge \mathrm{d}\bar{z}, \ldots, \omega_n=\frac{{\rm i}}{2}e^{u_n}\mathrm{d}z\wedge \mathrm{d}\bar{z}\right),
\end{equation}
together with the corresponding vector $\mathrm{Ric}\big(\vv{\omega}\big)$  of Ricci $(1,1)$-forms derived from the metric components of $\vv{\omega}$,
\[
\mathrm{Ric}\big(\vv{\omega}\big):=\left(\mathrm{Ric}(\omega_1)=-{\rm i}\partial \bar{\partial} u_1, \ldots, \mathrm{Ric}(\omega_n)=-{\rm i}\partial \bar{\partial} u_n\right).
\]
We define $\vv{\omega}$ as a \textit{solution} to the $\mathrm{SU}(n+1)$ Toda system on $\mathfrak{X}\setminus\mathfrak{S}$ if it satisfies the equation
\begin{equation}
\label{equ:TodaSys}
\mathrm{Ric}(\vv{\omega})=\vv{\omega} \cdot (2a_{ij})_{n\times n}
\end{equation}
on $\mathfrak{X}\setminus\mathfrak{S}$, where $(a_{ij})$ is the Cartan matrix for $\mathfrak{su}(n+1)$, represented by
\begin{equation*}
(a_{ij})_{n\times n}=\begin{pmatrix}
        2 & -1 & 0 & \ldots  & \ldots  & 0\\
        -1 & 2 & -1 & 0 & \ldots & 0\\
        0 & -1 & 2 & -1 & \ldots & 0\\
        \vdots & \vdots & \vdots & \vdots & \ddots & \vdots\\
        0 & \ldots & 0 & -1 & 2 & -1\\
        0 & \ldots & 0 & 0 & -1 & 2
\end{pmatrix}_{n\times n}.
\end{equation*}
This system, $\vv{\omega} = \mathrm{Ric}\big(\vv{\omega}\big) \cdot (2a_{ij})_{n\times n}$, is identified as the $\text{SU}(n+1)$ \textit{Toda system on the Riemann surface}, introducing a factor of $2$ for consistency with the notation in \cite{MSSX2024}. 
It is important to note that a solution \(\vv{\omega}\) to \eqref{equ:TodaSys} can exhibit wild behavior in the vicinity of punctures within \(\mathfrak{S}\). For instance, the integral of \(\omega_j\) in a neighborhood around a point in \(\mathfrak{S}\) may diverge. 
Since the work of Jost-Lin-Wang \cite[Proposition 3.1.]{JLW:2006}, there has been an ongoing investigation into solutions of Toda systems with  cone singularities. We will soon provide a precise definition for these solutions. 
To specify, we assign to each point $P \in \mathfrak{S}$ a vector $\vv{\gamma_P} = (\gamma_{P,1}, \ldots, \gamma_{P,n})$ of real numbers which are greater than $-1$ and do not vanish simultaneously, encapsulated into the $\mathbb{R}$-divisor vector $\vv{\mathfrak{D}}:=\big({\mathfrak D}_1,\ldots, {\mathfrak D}_n\big)$  with
$
\mathfrak{D}_1= \sum_{P \in \mathfrak{S}} \gamma_{P,1}[P], \ldots, \mathfrak{D}_n= \sum_{P \in \mathfrak{S}} \gamma_{P,n}[P]
$.
Denoting by $\delta_{\vv{\mathfrak D}}$ the following vector of $(1,1)$-currents \cite[Chapter 3]{GH:1994},
\begin{equation}
\label{equ:delta}
\delta_{\vv{\mathfrak D}}:=2\pi\left(\sum_{P \in \mathfrak{S}} \gamma_{P,1}\delta_P,\ldots, 
\sum_{P \in \mathfrak{S}} \gamma_{P,n}\delta_P\right),
\end{equation}
we give the following:

\begin{definition}
\label{def:sol_cone}
{\rm 
We consider $\vv{\omega}$ 
a  {\it solution} to the {\it $\mathrm{SU}(n+1)$ Toda system on $\mathfrak{X}$ with cone singularities} $\vv{\mathfrak{D}}$
(i.e., {\it representing $\vv{\mathfrak{D}}$}) if it has finite area over each compact subset $K$ of $\mathfrak{X}$, i.e. $\int_K\,\omega_j<\infty$ for all $1\leq j\leq n$, 
and satisfies the system
\begin{equation}
\label{equ:TodaSysCone}
\mathrm{Ric}\big(\vv{\omega}\big)=-\delta_{\vv{\mathfrak D}}+\vv{\omega} \cdot (2a_{ij})_{n\times n}
\end{equation}
in the sense of $(1,1)$-current on $X$.  In particular, 
this system, when restricted to a sufficiently small chart $(U,z)$ around each $P\in \mathfrak{S}$, takes the following form:
\begin{equation*}
\frac{\rm i}{2}\Big(\partial\bar{\partial}u_i+\sum\limits_{j=1}^n a_{ij}e^{u_j}\mathrm{d}z\wedge \mathrm{d}\bar{z} \Big)=\pi\gamma_{P,i}\delta_P,\quad 1 \leq i \leq n.    
\end{equation*}
By \cite[Theorem 1.2. (ii)]{MSSX2024}, for each $1\leq i\leq n$, the component $\omega_i=\frac{{\rm i}}{2}e^{u_i}\mathrm{d}z\wedge \mathrm{d}\bar{z}$ of a solution $\vv{\omega}$ to \eqref{equ:TodaSysCone} forms a cone K\"ahler metric on $X$ with cone angle $2\pi(1+\gamma_{P, i})$ at $P\in \mathfrak{S}$. In this context, we say that $\vv{\omega}$ is a {\it solution with cone singularities} $\vv{\mathfrak{D}}$ (i.e., \textit{representing $\vv{\mathfrak{D}}$}). }
\end{definition}

\noindent It is noteworthy that when \(n=1\), the term \(\vv{\omega}=\omega_1\) denotes a cone spherical metric that represents the divisor \(\mathfrak{D}_1\) on $\mathfrak{X}$. For further insights into cone spherical metrics, see \cite{LX2023} and the references contained therein.
Unlike the conventional definition of the Toda system with singularities,  our intrinsic approach does not rely on a predetermined background K\"ahler metric and situates the problem within the context of Kähler Geometry. 
Specifically, our ongoing project is dedicated to investigating vectors of K\"ahler metrics with cone singularities along divisors \cite{Don2012} as potential solutions to the Toda systems associated with any complex simple Lie algebra on K\"ahler manifolds. This approach naturally extends both the Toda system \eqref{equ:TodaSysCone} on Riemann surfaces and the Monge-Ampère equation concerning cone K\"ahler-Einstein metrics with positive scalar curvatures on K\"ahler manifolds of dimension $\geq 2$ (\cite{Don2012}).

\subsection{A basic correspondence}
\label{subsec:corr}
%Brief description of the correspondence between solutions and curves.
Drawing upon the  work of Gervais-Matsuo \cite[Section 2.2.]{GM:1993}, Jost-Wang \cite[Section 3]{JW:2002}, and \cite[Section 2]{MSSX2024}, in this
subsection, we shall present a concise overview of the basic correspondence between solutions to the ${\rm SU}(n+1)$ Toda systems with cone singularities $\vv{\mathfrak{D}}$ on $\mathfrak{X}$, and totally unramified unitary curves on $\mathfrak{X}\setminus \mathfrak{S}$ and with regular singularities $\vv{\mathfrak{D}}$. This overview lays a robust foundation for articulating the main results of our manuscript and situating them within the context of classical findings such as \cite{GM:1993, JW:2002, LWY:2012, LNW2018}. 

A \textit{totally unramified unitary curve} $f$ on $\mathfrak{X}\setminus \mathfrak{S}$ is a multi-valued holomorphic map $f:\mathfrak{X}\setminus \mathfrak{S}\to \mathbb{P}^n$, characterized by the following properties:
\begin{itemize}
    \item The monodromy of $f$ resides within ${\rm PSU}(n+1)$, the group of holomorphic isometries preserving the Fubini-Study metric $\omega_{\rm FS}$ on $\mathbb{P}^n$.
    \item At each point $z \in \mathfrak{X}\setminus \mathfrak{S}$, any germ $\mathfrak{f}$ of $f$ at $z$ is totally unramified and, notably, non-degenerate, near $z$.
\end{itemize}
For a more detailed exposition of  this concept, please refer to \cite[Definition 2.1.]{MSSX2024}. 
%Such a curve is called  {\it toric} if, and only if, its monodromy is contained in a maximal torus of ${\rm PSU}(n+1)$.

\begin{definition}
\label{def:curve_reg_sing}  
{\rm 
We define a totally unramified unitary curve $f:\mathfrak{X}\setminus \mathfrak{S}\to\mathbb{P}^n$ to have {\it regular singularities} 
$$\vv{\mathfrak{D}}=\left(\sum_{P\in \mathfrak{S}}\,\gamma_{P,1}[P],\ldots, \sum_{P\in \mathfrak{S}}\,\gamma_{P,n}[P]\right)$$
i.e., {\it representing} $\vv{\mathfrak{D}}$, if, for each point $P \in \mathfrak{S}$, there exists an element $\varphi \in {\rm PSU}(n+1)$ such that the composition $\varphi \circ f$ when restricted to a punctured disk $\{0 < |z| < 1\}\subset\mathfrak{X}\setminus \mathfrak{S}$ around $P$ can be expressed as
\begin{equation}
\label{equ:reg_sing}
\varphi\circ f(z) = \left[z^{\beta_{P,0}}g_0(z): z^{\beta_{P,1}}g_1(z): \ldots : z^{\beta_{P,n}}g_n(z)\right],
\end{equation}
where the functions $g_0(z), \ldots, g_n(z)$ are holomorphic and nonvanishing at $z=0$, i.e., $P$. The exponents $\beta_{P,0}, \ldots, \beta_{P,n}$ are real numbers satisfying 
\begin{equation}
\label{equ:gamma_to_beta}
\beta_{P,j}-\beta_{P,j-1}=\gamma_{P,j}+1\quad \text{for all}\quad 1\leq j\leq n.
\end{equation}
In this context, given a point $P$ in $\mathfrak{S}$, if all $\gamma_{P,j}$'s are non-negative integers and do not vanish simultaneously, then $P \in \mathfrak{S}$ is referred to as a {\it ramification point} of $f$  (\cite[pp.\,266-268]{GH:1994}); if at least one of $\{\gamma_{P,j}\}_{j=1}^n$ is non-integer, then $P$ is called a {\rm branch point} of $f$. Both ramification points and branch points are called {\it regular singularities} of $f$.
It is important to note that  $f$ is totally unramified at a point $P \in X$ if and only if $\gamma_{P,j} = 0$ for all $1 \leq j \leq n$. 
}
\end{definition}

From a totally unramified unitary curve $f:\mathfrak{X}\setminus \mathfrak{S}\to\mathbb{P}^n$, one can derive a solution to the ${\rm SU}(n+1)$ Toda system \eqref{equ:TodaSys} on $\mathfrak{X} \setminus \mathfrak{S}$. In fact, for $0 \leq j \leq n-1$, the $j$-th associated curve 
$f_j : \mathfrak{X} \setminus \mathfrak{S} \to G(j+1, n+1) \subset \mathbb{P}\left(\Lambda^{j+1}\mathbb{C}^{n+1}\right)$
of $f$ is also a unitary curve (\cite[pp.\,263-264]{GH:1994} and \cite[Definition 2.1.]{MSSX2024}). For simplicity, 
we uniformly adopt the symbol $\omega_{\rm FS}$ to represent the Fubini-Study metrics on $\mathbb{P}\left(\Lambda^{j+1}\mathbb{C}^{n+1}\right)$ for all $0 \leq j \leq n-1$.
By employing the infinitesimal Pl\"ucker formula, 
$f : \mathfrak{X} \setminus \mathfrak{S} \to \mathbb{P}^n$ 
produces a vector  of K\"ahler metrics
\begin{equation}
\label{equ:curve_to_sol}  
\vv{\omega} = \left(f_0^*\omega_{\rm FS}, \ldots, f_{n-1}^*\omega_{\rm FS}\right),
\end{equation} 
which serves as a solution to \eqref{equ:TodaSys} on $\mathfrak{X} \setminus \mathfrak{S}$, as delineated in \cite[Lemma 2.2]{MSSX2024}. Conversely, each solution $\vv{\omega}$ to the ${\rm SU}(n+1)$ Toda system \eqref{equ:TodaSys} on $\mathfrak{X} \setminus \mathfrak{S}$ gives rise to
a series of totally unramified unitary curves $\mathfrak{X}\setminus \mathfrak{S}\to\mathbb{P}^n$, where any two are distinguishable by
a post-composition of an element in ${\rm PSU}(n+1)$. Furthermore, these curves reconstruct the solution $\omega$ as defined in \eqref{equ:curve_to_sol}, and they are termed {\it curves associated with} $\omega$. The intricacies will be expounded in Lemma \ref{lem:corr}. 
%Such a solution is called {\it toric} if, and only if, its associated curves are toric. 
Therefore, we have finalized the exposition detailing the correspondence between solutions to the ${\rm SU}(n+1)$ Toda system, as described in \eqref{equ:TodaSys}, and totally unramified unitary curves  on $\mathfrak{X} \setminus \mathfrak{S}$. In Theorem \ref{thm:corr_sing}, we further elucidate a refined correspondence between solutions to the ${\rm SU}(n+1)$ Toda system \eqref{equ:TodaSysCone}  with cone singularities $\vv{\mathfrak{D}}$ on $\mathfrak{X}$, and totally unramified unitary curves on $\mathfrak{X} \setminus \mathfrak{S}$ with regular singularities $\vv{\mathfrak{D}}$.

\subsection{Exploring the concepts of toric curves and toric solutions}

We introduce the following innovative concept associated with the \(\mathrm{SU}(n+1)\) Toda system, inspired by its \(n=1\) scenario --- specifically, the reducible cone spherical metric \cite{UY:2000, CWWX:2015, MP:2015, Eremenko:2020}.

\begin{definition}
\label{def:toric}{\rm 
A \textit{toric curve} \(f: \mathfrak{X} \setminus \mathfrak{S}\) is defined as a totally unramified unitary curve whose monodromy resides within a maximal torus of \(\mathrm{PSU}(n+1)\). A solution \(\omega\) to the \(\mathrm{SU}(n+1)\) Toda system on \(\mathfrak{X} \setminus \mathfrak{S}\) is termed \textit{toric} if it yields toric associated curves. This leads to the concepts of a \textit{toric curve with regular singularities} $\vv{\mathfrak{D}}$ and a \textit{toric solution  with cone singularities} $\vv{\mathfrak{D}}$, respectively.}
\end{definition}

Having established the foundation for our methodology in addressing Toda systems, now is an opportune moment to review the classifications of solutions and the significant findings regarding their existence.
The pioneering effort by Jost-Wang \cite{JW:2002} in 2002 revealed that curves associated with finite-area solutions to the ${\rm SU}(n+1)$ Toda system on the complex plane $\mathbb{C}$  extend to rational normal curves from $\mathbb{P}^1$ to $\mathbb{P}^n$. By this, they obtained a thorough classification of solutions on $\mathbb{P}^1$. In 2007, Eremenko \cite[Theorem 2]{Ere:2007} further classified curves linked to solutions expanding polynomially in area at $\infty$ of $\mathbb{C}$.
Continuing this trajectory, Lin-Wei-Ye \cite{LWY:2012}, five years on, managed to classify all solutions to the ${\rm SU}(n+1)$ system with two cone singularities on $\mathbb{P}^1$. Simultaneously, they also characterized the corresponding cone singularities.
Building on this, quite recently, Karmakar-Lin-Nie-Wei \cite{KLNW2022}  broadened the scope to include any Toda system tied to a complex simple Lie algebra. 
The preceding solutions are all toric since the fundamental groups of the underlying surfaces are either trivial or isomorphic to 
${\Bbb Z}$. In 2018, Lin-Nie-Wei \cite{LNW2018} obtained some existence results about solutions to the ${\rm SU}(n+1)$ system 
with three cone singularities on the Riemann sphere. 
Between 2015 and 2016, Battaglia \cite{Batt_JMAA2015, Batt_Adv2015, Batt2016} laid down extensive theorems on the existence and absence of solutions for the ${\rm SU}(3)$ Toda system with cone singularities on compact Riemann surfaces. Following suit, Lin-Yang-Zhong \cite{LYZ2020} made significant strides by proving existence theorems for Toda systems related to Lie algebras of types $A_n$, $B_n$, $C_n$, and $G_2$, featuring cone singularities on compact Riemann surfaces with positive genera. More recently, Chen-Lin \cite{CL2022, CL2023} have unveiled numerous findings regarding the ${\rm SU}(3)$ Toda system with cone singularities on tori.
In their most recent collaboration \cite{MSSX2024}, Sun and the authors of this manuscript meticulously categorized the solutions to the \({\rm SU}(n+1)\) Toda system, delineated on the disk \({|z| < 1}\). These solutions are characterized by a cone singularity at \(z=0\), possess finite area, and inherently exhibit the toric property.

In the left part of this introductory section, we shall focus on toric solutions to the ${\rm SU}(n+1)$ Toda system with cone singularities 
\[\vv{D}=\left(D_1=\sum_{i=1}^k\, \gamma_{i,1}[P_i],\ldots, D_n=\sum_{i=1}^k\, \gamma_{i,n}[P_i]\right),\]
where $P_1,\ldots, P_k$ are distinct points 
on a {\it compact} Riemann surface $X$, 
and $\Big(\gamma_{i,j}\Big)_{\substack{1\leq i\leq k \\ 1\leq j\leq n}}$ is a $k\times n$ matrix of real numbers greater than $-1$ such that for each $1\leq i\leq k$, at least one of $\{\gamma_{ij}\}_{1\leq j\leq n}$ does not vanish.   
We call $\{P_1,\cdots, P_k\}$ the {\it support}, and $\Big(\gamma_{i,j}\Big)_{\substack{1\leq i\leq k \\ 1\leq j\leq n}}$ the {\it coefficient matrix} of $\vv{D}$.
Chen, Wang, Wu, and the last author \cite[Theorems 1.4. and 1.5.]{CWWX:2015} established on \(X\) a correspondence between toric solutions to the \(\mathrm{SU}(2)\) Toda system with cone singularities, namely reducible cone spherical metrics, and meromorphic one-forms with simple poles and purely imaginary periods on \(X\). This correspondence is particularly noteworthy as it allows for the explicit characterization of the singularity information of a reducible metric in terms of the corresponding one-form, which was referred to as the \textit{character one-form} of the metric therein. 
Moreover, meromorphic one-forms with simple poles and purely imaginary periods are plentifully available on Riemann surfaces, cf. \cite[$\S$15]{Weyl1964}, \cite[$\S$8-1]{Springer1981} and \cite[$\S$II.4-5]{FK1992} for precise statements and their proof. In particular,  the explicit formulation of such one-forms was meticulously documented in \cite[Example 4.7.]{CWWX:2015} on the Riemann sphere. 
To generalize this correspondence for toric solutions to  ${\rm SU}(n+1)$ Toda system with cone singularities on $X$,
we introduce the following:

\begin{definition}
\label{def:ensemble}
{\rm 
A \textit{character $n$-ensemble} on \(X\) is defined as a vector \(\vv{\Omega} = (\Omega_1, \ldots, \Omega_n)\), composed of $n$ meromorphic one-forms with simple poles and purely imaginary periods on \(X\) such that there exist finitely many points \(P_1, \ldots, P_k\) on \(X\) and 
\begin{equation}
\label{equ:ensemble}
f_{\vv{\Omega},\vv{\rho}}(z) := \left[1: \rho_1\cdot\exp\left(\int^z \Omega_1\right): \ldots :\rho_n\cdot\exp\left(\int^z \Omega_n\right)\right]
\end{equation}
generates a family of totally unramified unitary curves mapping from \(X \setminus \{P_1, \ldots, P_k\}\) to \(\mathbb{P}^n\), where
$\vv{\rho}=(\rho_1,\ldots,\rho_n)$'s vary over all vectors of $n$ positive real numbers. Importantly, these curves exhibit monodromy within the diagonal maximal torus \(\mathbb{T}^n\) of ${\rm PSU}(n+1)$, specified by
\begin{equation*}
    \mathbb{T}^n := \left\{ \varphi([z_0: \ldots: z_n]) = \left[e^{\mathrm{i}\theta_0}z_0: \ldots: e^{\mathrm{i}\theta_n}z_n\right] \, \middle| \, (\theta_0, \ldots, \theta_n) \in \mathbb{R}^{n+1} \right\}.
\end{equation*}
Hence, they are inherently toric curves from \(X \setminus \{P_1, \ldots, P_k\}\) to \(\mathbb{P}^n\).
Within the given context, it can be demonstrated that the preceding curves described inherently possess the same regular singularities, denoted by $\vv{D}$. These singularities are characterized by a support $\{P_1, \cdots, P_k\}$ and their coefficient matrix 
$\Big(\gamma_{i,j}\Big)_{\substack{1\leq i\leq k\\ 1\leq j\leq n}}$, both of which are 
implicitly determined by the $n$-ensemble $\vv{\Omega}$ (Proposition \ref{prop:ensemble} and Lemma \ref{lem:n-tuple}).}
\end{definition}

\subsection{Main results}

We establish on $X$ the following general correspondence for the ${\rm SU}(n+1)$ Toda system which puts the one in \cite{CWWX:2015} as its $n=1$ scenario.

\begin{theorem}
\label{thm:corr}
There exists a correspondence between character $n$-ensembles and toric solutions \(\omega\) to the \(\mathrm{PSU}(n+1)\) Toda system with cone singularities  on a compact Riemann surface \(X\). Specifically, the following two statements hold:
\begin{enumerate}
    \item Given a character $n$-ensemble \(\vv{\Omega}\) on \(X\), it induces a family of toric curves via \eqref{equ:ensemble} with 
    regular singularities $\vv{D}$ on \(X\), which in turn defines a corresponding family of toric solutions representing $\vv{D}$,  via \eqref{equ:curve_to_sol}.
    \item Given a toric solution \(\vv{\omega}\) with cone singularities \(\vv{D}\) on \(X\), it has an associated toric curve \(f=[f_0,\ldots, f_n]\) with monodromy in $\mathbb{T}^n$ and with regular singularities \(\vv{D}\) on \(X\) such that 
    \[\Omega:=\left(\mathrm{d}\left[\log \frac{f_1}{f_0}\right],\ldots, \mathrm{d}\left[\log \frac{f_n}{f_0}\right]  \right)\]
    forms a character $n$-ensemble on \(X\).
\end{enumerate}
\end{theorem}

%State the main theorem: correspondence between character ensembles and toric solutions
%Describe the new class of solutions and compare them with LYZ.
%In section 3, through the analytical properties of the 1-form group, we discuss the singularity information of %Toric% curves and ultimately obtain a key theorem:

A natural question arises regarding the characterization of regular singularities on curves generated by a character ensemble, as defined in \eqref{equ:ensemble}. The following theorem offers a partial response to this inquiry.

\begin{theorem} 
\label{thm:sing}
Let $\vv{\Omega}$ be a character $n$-ensemble and $f$ one in the family of curves  generated by $\vv{\Omega}$ in terms of \eqref{equ:ensemble} on a compact Riemann surface $X$. Denote by $\mathcal{P}$ the set of poles of all $\Omega_j$'s. There holds the following statements{\rm :}

\begin{itemize}
\item[(1)] A point $P$ on $X$ is a branch point of $f$ if and only if it is a pole of some component $\Omega_j$ of $\vv{\Omega}$ such that the residue of $\Omega_j$ at $P$ is noninteger.  

\item[(2)] A zero of $\vv{\Omega}$, at which all components of $\vv{\Omega}$ vanish, is a ramification point of $f$. 
 
\item[(3)] An algorithm is presented that identifies the regular singularities \(\vv{D}\) of \(f\) in finitely many steps. Its details will be provided within the proof.
\end{itemize}

\end{theorem} 

Below, we present two examples of toric curves with regular singularities that yield novel toric solutions to the \({\rm SU}(n+1)\) Toda systems with cone singularities. These examples extend beyond the scope of quite recent existence results presented in \cite[Theorems 1.8-9.]{LYZ2020}.

\begin{example}
\label{exam:Omega}
Consider a compact Riemann surface \(X\) and let \(\Omega\) be a nontrivial meromorphic one-form on \(X\) that has simple poles and purely imaginary periods. Define the vector of scaled one-forms 
\[
\vv{\Omega} := (\lambda_1 \Omega, \ldots, \lambda_n \Omega),
\]
where \(\lambda_1, \ldots, \lambda_n\) are distinct nonzero real numbers. This vector \(\vv{\Omega}\) constitutes a character \(n\)-ensemble on \(X\) and generates a family of toric solutions that are parametrized by \((\mathbb{R}_{>0})^n\) as per Equation \eqref{equ:ensemble}.

Furthermore, the singularities of these toric solutions can be fully characterized by \(\vv{\Omega}\). In particular, all the cone singularities of the solutions are confined to the set of zeroes and poles of \(\Omega\). Specifically, at a zero \(\mathfrak{q}\) of \(\Omega\), each metric component in the toric solution exhibits a cone angle of 
$
2\pi \left(1 + \operatorname{ord}_{\mathfrak{q}} \Omega\right).
$
\end{example}

%\begin{example}
%\label{exam:RS1}
%We construct on the Riemann sphere a family of toric curves to $\mathbb{P}^2$ with exactly three regular singularities, all of which serve as  %branch points.
%\end{example}

\begin{example}
\label{exam:RS1}  
Identify $\mathbb{P}^1$ with the Riemann sphere $\mathbb{C}\cup\infty$.
Consider the vector $(\gamma_1, \dots, \gamma_n) \in (-1, \infty)^n \setminus \mathbb{Z}^n$ and 
a positive integer $m$. Given $m$ distinct points $z_1, \dots, z_m$ on $\mathbb{C}\setminus \{0\}$ and
a matrix of non-negative integers $\Big(\gamma_{i,j}\Big)_{\substack{1 \leq i \leq m \\ 1 \leq j \leq n}}$, there exists a family of toric curves 
from $\mathbb{P}^1\setminus\{0,\infty,z_1,\ldots,z_m\}$ to $\mathbb{P}^n$ 
parametrized by $\big(\mathbb{R}_{>0}\big)^n$ satisfying the following properties:
\begin{itemize}
    \item All the curves have the same regular singularities. In particular, they have exactly two branch points at $0$ and $\infty$, with $\gamma_{[0],j} = \gamma_j$ for all $1 \leq j \leq n$;
    \item The points $z_1, \dots, z_m$ are all ramification points of them, where $\gamma_{[z_i],j} = \gamma_{i,j}$ for all $1 \leq i \leq m$ and $1 \leq j \leq n$.
\end{itemize}   
Moreover,  the vector $\left(\gamma_{[\infty],j}\right)_{j=1}^n$ has finitely many choices, bounded from above by 
\[
    \binom{\sum_{i=1}^m \big(n \gamma_{i,1} + (n-1) \gamma_{i,2} + \cdots + \gamma_{i,n}\big) + n + 1}{n}.
    \]
\end{example}

%%%%%%%%%%%%%%%%%%%%% 
\noindent {\bf Outline.}
The introduction concludes with an outline of the structure for the remainder of this manuscript. 
Section 2 focuses on establishing the relationship between curves with regular singularities 
and solutions with cone singularities, as detailed in Theorem \ref{thm:corr_sing}. 
The proof of Theorem \ref{thm:sing} is presented in Section 3, followed by a detailed substantiation 
of Theorem \ref{thm:corr} in Section 4. Additionally, Section 3 introduces a new concept of a 
non-degenerate $n$-tuple of one-forms. This concept is simpler yet equivalent to the character 
$n$-ensemble, as demonstrated in Proposition \ref{prop:ensemble}.
Section 5 provides a comprehensive discussion of the two preceding examples.
In the final section, we propose three open questions to guide further investigation.

%The first is a conjecture concerning smooth solutions to the ${\rm SU}(n+1)$ Toda system on $\mathbb{P}^m$ for $m,n > 1$. The second involves exploring the relationship between solutions with cone singularities on compact Riemann surfaces and line sub-bundles of polystable parabolic vector bundles of rank $(n+1)$.

\section{Correspondence between curves and solutions}
In this section, we shall prove a lemma and a theorem that sequentially facilitate the establishment of the basic correspondence between solutions to the ${\rm SU}(n+1)$ Toda systems with cone singularities and unitary curves with regular singularities. This correspondence was initially outlined in Subsection \ref{subsec:corr}. Throughout, we will consistently use the notations previously introduced.

\begin{lemma}
\label{lem:corr}
\begin{itemize}
    \item[(1)] \textbf{From Curve to Solution{\rm :}} A totally unramified unitary curve \(f:\mathfrak{X}\setminus\mathfrak{S} \rightarrow \mathbb{P}^n\) induces a solution \(\vv{\omega} := \big(f_0^*\omega_{\rm FS}, \ldots, f_{n-1}^*\omega_{\rm FS}\big)\) to the \({\rm SU}(n+1)\) Toda system on \(\mathfrak{X}\setminus\mathfrak{S}\). Moreover, any curve \( \varphi \circ f \), with \(\varphi\) in \({\rm PSU}(n+1)\), produces the same solution as \(f\).
    \item[(2)] \textbf{From Solution to Curve{\rm :}} Every solution \(\vv{\omega}\) to the \({\rm SU}(n+1)\) Toda system on \(\mathfrak{X}\setminus\mathfrak{S}\) corresponds to at least one totally unramified unitary curve \(f:\mathfrak{X}\setminus\mathfrak{S}\rightarrow \mathbb{P}^n\) that generates \(\vv{\omega}\). Furthermore, any other curve that corresponds to \(\vv{\omega}\) will be in the form of \( \varphi \circ f \) for some \(\varphi\) in \({\rm PSU}(n+1)\).
\end{itemize}
\end{lemma}

\begin{proof}

\noindent (1) Consider a totally unramified unitary curve \(f : \mathfrak{X} \setminus \mathfrak{S} \to \mathbb{P}^n\). For the definition of the $j$-th associated function \(f_j\) where \(0 \leq j \leq n\), we refer to \cite[pp.\,263-264]{GH:1994} and \cite[Definition 2.1]{MSSX2024}. The proof of the first statement closely mirrors that of Lemma 2.2 in \cite{MSSX2024}, which specifically addresses the plane domain scenario.
The second statement follows from the fact that each element $\varphi$ of \({\rm PSU}(n+1)\) preserves the Fubini-Study metric on 
\(\mathbb{P}^n\).

(2) Consider a solution $\vv{\omega} = (\omega_1, \ldots, \omega_n)$ to the ${\rm SU}(n+1)$ Toda system on $\mathfrak{X} \setminus \mathfrak{S}$. Fix a point $\mathfrak{p}$ on $\mathfrak{X} \setminus \mathfrak{S}$. Restricting $\vv{\omega}$ to a disc chart $(D, |z|<1)$ around $\mathfrak{p}$, represented as $\vv{\omega}|_D$, results in a solution over $D$. Employing \cite[Lemma 2.3.]{MSSX2024}, this restriction leads to the formation of a totally unramified holomorphic curve $\mathfrak{f}_{\mathfrak{p}}: D \to \mathbb{P}^n$, which specifically induces $\vv{\omega}|_D$. Notably, the Fubini-Study form pullback via $\mathfrak{f}_{\mathfrak{p}}$, denoted $\mathfrak{f}_{\mathfrak{p}}^* \omega_{\rm FS}$, coincides with the restricted $\omega_1$ on $D$.
According to the local rigidity theorem posited by Eugenio Calabi (refer to \cite[Theorem 9]{Calabi1953} and \cite[(4.12)]{Griffiths:1974}), the curve $\mathfrak{f}_{\mathfrak{p}}$ is uniquely determined by $\vv{\omega}$, up to a transformation in ${\rm PSU}(n+1)$. Further, leveraging the developing map concept outlined in \cite[$\S$3.4]{Thurston1997}, a unique, totally unramified unitary curve $f: \mathfrak{X} \setminus \mathfrak{S} \to \mathbb{P}^n$ can be realized through analytic continuations of $\mathfrak{f}_{\mathfrak{p}}$ along curves on $\mathfrak{X} \setminus \mathfrak{S}$. Its monodromy results in a group homomorphism 
$$\mathcal{M}_f:\pi_1\big(\mathfrak{X} \setminus \mathfrak{S},\,\mathfrak{p}\bigr)\to {\rm PSU}(n+1).$$
This curve $f$ not only induces $\vv{\omega}$ as initially described in (1) but is also uniquely characterized by the same local rigidity theorem within the confines of ${\rm PSU}(n+1)$.
\end{proof}

The lemma above can be generalized to include cases with cone (regular) singularities as follows:

\begin{theorem}
\label{thm:corr_sing}
We use the notations from \textnormal{Definitions~\ref{def:sol_cone} and \ref{def:curve_reg_sing}}.
\begin{itemize}
    \item[(1)] \textbf{From curve with regular singularities to solution with cone singularities\,{\rm:}}
    A totally unramified unitary curve \(f : \mathfrak{X} \setminus \mathfrak{S} \to \mathbb{P}^n\) with regular singularities \(\vv{\mathfrak{D}}\)
    induces a solution \(\vv{\omega} = (f_0^* \omega_{\textnormal{FS}}, \ldots, f_{n-1}^* \omega_{\textnormal{FS}})\) to the \(\textnormal{SU}(n+1)\) Toda system
    on \(\mathfrak{X}\) with cone singularities \(\vv{\mathfrak{D}}\). Moreover, any curve \( \varphi \circ f \), where \(\varphi\) is in \(\textnormal{PSU}(n+1)\), produces the same solution as \(f\).

    \item[(2)] \textbf{From solution with cone singularities to curve with regular singularities\,{\rm :}}
    Every solution \(\vv{\omega}\) to the \(\textnormal{SU}(n+1)\) Toda system on \(\mathfrak{X}\) with cone singularities \(\vv{\mathfrak{D}}\)
    corresponds to at least one totally unramified unitary curve \(f : \mathfrak{X} \setminus \mathfrak{S} \to \mathbb{P}^n\) with regular singularities \(\vv{\mathfrak{D}}\) that generates \(\vv{\omega}\). We call $f$ a {\rm curve associated with} $\vv{\omega}$. Furthermore, any other curve that corresponds to \(\vv{\omega}\) will be in the form of \( \varphi \circ f \) for some \(\varphi\) in \(\textnormal{PSU}(n+1)\). 
\end{itemize}
\end{theorem}

\begin{proof}
By Lemma~\ref{lem:corr}, it is sufficient to verify the properties concerning the singularities \(\vv{D}\).

\begin{itemize}
    \item[(1)] Consider a totally unramified unitary curve \(f : \mathfrak{X} \setminus \mathfrak{S} \to \mathbb{P}^n\) with regular singularities \(\vv{\mathfrak{D}}\). According to Lemma~\ref{lem:corr}(1), this curve induces a solution \(\vv{\omega}\) on \(\mathfrak{X} \setminus \mathfrak{S}\). Furthermore, based on \cite[Formula 3.]{MSSX2024} and the infinitesimal Pl\" ucker formula, \(\vv{\omega}\) exhibits cone singularities \(\vv{\mathfrak{D}}\).

    \item[(2)] Consider a solution \(\vv{\omega} = (\omega_1, \ldots, \omega_n)\) on \(\mathfrak{X}\) with cone singularities \(\vv{D}\). Lemma~\ref{lem:corr}(2) states that this solution corresponds to a totally unramified unitary curve \(f : \mathfrak{X} \setminus \mathfrak{S} \to \mathbb{P}^n\). Moreover, according to \cite[Theorem 1.2.(i)]{MSSX2024}, curve \(f\) has regular singularities \(\vv{D}\).
\end{itemize}
\end{proof}

\section{Character ensembles and toric curves with regular singularities}

In this section, we shall prove Theorem \ref{thm:sing}. To this end,  we demonstrate in Proposition \ref{prop:ensemble} that character $n$-ensembles correspond to $n$-tuples of meromorphic one-forms with simple poles and purely imaginary periods. These $n$-tuples are {\it non-degenerate} in the sense that they define non-degenerate curves under equation \eqref{equ:ensemble}, as established in Definition \ref{def:nondeg_forms}. 
This explanation renders the concept of a character ensemble much more accessible and relatable.

\subsection{A non-degenerate $n$-tuple of one-forms is a character $n$-ensemble}

\begin{definition}
\label{def:nondeg_forms}
{\rm 
Let \(\vv{\Omega}=(\Omega_1, \ldots, \Omega_n)\) be an \(n\)-tuple of meromorphic one-forms on a compact Riemann surface \(X\), characterized by having simple poles and purely imaginary periods. We define \(\vv{\Omega}\) to be {\it non-degenerate} if there exists a point \(\mathfrak{p}\) in \(X\) such that:
\begin{itemize}
    \item All components of \(\vv{\Omega}\) are holomorphic at \(\mathfrak{p}\).
    \item In a disc chart \((D, |z|<1)\) around \(\mathfrak{p}\), i.e., $z=0$,  the map
    \begin{equation}
    \label{equ:n-tuple}
        \mathfrak{f_p}(z) := \left[1: \exp\left(\int_0^z \Omega_1\right): \ldots : \exp\left(\int_0^z \Omega_n\right)\right]
    \end{equation}
    defines a holomorphic curve from \(D\) to \(\mathbb{P}^n\) that is non-degenerate, meaning that its image is not contained within any hyperplane of \(\mathbb{P}^n\).
\end{itemize}
It is noteworthy that in this context, the curve \(\mathfrak{f_p}(z)\) remains non-degenerate near every point \(\mathfrak{p}\), where all components of \(\vv{\Omega}\) are holomorphic. 
Denote by $\mathcal{P}$ the set of poles of all $\Omega_j$'s. 
Similar to Equation \eqref{equ:ensemble}, \(\vv{\Omega}\) induces a family of non-degenerate, multi-valued holomorphic curves 
\begin{equation}
\label{equ:n-tuple}
\left\{f_{\vv{\Omega}, \vec{\rho}}: \vv{\rho}\in \big(\mathbb{R}_{>0}\big)^n\right\}
\end{equation}
mapping from \(X \setminus \mathcal{P}\) to \(\mathbb{P}^n\). These curves share identical monodromy in \(\mathbb{T}^n\) and possess the same regular singularities on $X$. }
\end{definition} 

\begin{lemma}
\label{lem:n-tuple}
Adopting the context of Definition \ref{def:nondeg_forms} and considering a fixed vector \(\vec{\rho}\) consisting of \(n\) positive numbers, we have the following properties for the curve  \(f_{\vec{\Omega}, \vec{\rho}}\){\rm :}
\begin{enumerate}
    \item A point \(\mathfrak{p}\) on \(X\) is identified as a branch point of the curve if and only if \(\mathfrak{p}\) is a pole of some \(\Omega_j\) where the residue \({\rm Res}_{\mathfrak{p}}\Omega_j\) is not an integer.
    \item The curve is totally unramified on $X$, except at a finite number of regular singularities.
\end{enumerate}
\end{lemma}
\begin{proof}
(1) Take a disc chart $(D,|z|<1)$ around $\mathfrak{p}$ and denote by $\vv{\mathfrak{r}}=\mathfrak{(r_1,\ldots, r_n)}\in \mathbb{R}^n$ the vector of residues of $\vv{\Omega}=\left(\Omega_1,\ldots,\Omega_n\right)$ at $\mathfrak{p}$. We choose a germ $\mathfrak{f_p}$ of 
\(f_{\vec{\Omega}, \vec{\rho}}\) in $D^*:=\{0<|z|<1\}$ with form 
\[\left[1:z^{\mathfrak{r_1}}h_1(z):\ldots:z^{\mathfrak{r_n}}h_n(z)\right]=\left[z^0:z^{\mathfrak{r_1}}h_1(z):\ldots:z^{\mathfrak{r_n}}h_n(z)\right], \]
where each $h_j$ is holomorphic in $D$ and does not vanish at $z=0$.

Assume that the residues $\mathfrak{r}_j$ are all integers. In this case, the curve \(f_{\vec{\Omega}, \vec{\rho}}\) can be extended holomorphically to the point $\mathfrak{p}$. Therefore, at $\mathfrak{p}$, the curve is either totally unramified or it is a ramification point. Specifically, $\mathfrak{p}$ is not a branch point. Furthermore, the ramification indices $\gamma_{\mathfrak{p}, j}$ for $1 \leq j \leq n$ can be calculated using the algorithm described on pages 266-268 in \cite{GH:1994}.

Suppose that at least one of the residues $\mathfrak{r}_j$ is not an integer. Consequently, 
the set $\{0, \mathfrak{r}_1, \cdots, \mathfrak{r}_n\}$ is partitioned into at least two equivalence classes under the modulo $\mathbb{Z}$ equivalence relation.
According to the argument used in the proof in \cite[Theorem 3.1.]{MSSX2024}, there exists a transformation $\varphi \in \mathrm{PSU}(n+1)$ such that $\varphi(\mathfrak{f_p})$ is represented by
\begin{equation}
\label{equ:quasi}
\left[z^{b_0} : z^{b_1} H_1(z) : \ldots : z^{b_n} H_n(z)\right],
\end{equation}
where $b_0 < \ldots < b_n$, and each function $H_j(z)$ is holomorphic in $D$ and non-vanishing at $z=0$. Additionally, there is at least one index $j$ (where $0 \leq j \leq n-1$) for which $b_0 \equiv \ldots \equiv b_j \pmod{\mathbb{Z}}$, and 
$\gamma_{\mathfrak{p}, j+1} = b_{j+1} - b_j - 1\notin \mathbb{Z}$. Therefore, $\mathfrak{p}$ is identified as a branch point of the curve \(f_{\vec{\Omega}, \vec{\rho}}\). Notably, we could  read all $\gamma_{\mathfrak{p}, \cdot}$ indices as
\[\gamma_{\mathfrak{p},j}=b_{j+1}-b_j-1\quad {\rm for\  all}\quad 1\leq j\leq n\]
from \eqref{equ:quasi}, called  the {\it quasi-canonical form} of $f$.

(2)  We aim to show that the function $f_{\vec{\Omega}, \vec{\rho}}$ has finitely many ramification points on a compact Riemann surface $X$. Assuming the contrary, let $\{q_n\}$ be a sequence of distinct ramification points converging to some point $\mathfrak{p}$ on $X$. We consider two cases based on the location of $\mathfrak{p}$:

\begin{itemize}
    \item[(i)] \textbf{Case $\mathfrak{p} \notin \mathcal{P}$:} 
  Select a disc chart $(D, |z| < 1)$ around $\mathfrak{p}$, which includes $\{q_n\}$ and excludes $\mathcal{P}$. Consider a germ $\mathfrak{f}_p$ of $f$ in $D$ and define the $n$-th associated curve $\Lambda_n(\mathfrak{f_p}, z) = \mathfrak{f_p}(z) \wedge \mathfrak{f_p}'(z) \wedge \ldots \wedge \mathfrak{f_p}^{(n)}(z)$, mapping $D$ to $\Lambda^{n+1}(\mathbb{C}^{n+1})$. According to computations in \cite{GH:1994}, $\Lambda_n(\mathfrak{f_p}, z)$ vanishes at $\{q_n\} \cup \{\mathfrak{p}\}$ and thus identically on $D$. This implies that $\Lambda_n(f)$, and consequently $f$, is degenerate on $X \setminus \mathcal{P}$. This is a contradiction since both $\vv{\Omega}$ and $f$ are non-degenerate. 

    \item[(ii)] \textbf{Case $\mathfrak{p} \in \mathcal{P}$:} 
   Using a similar setup with a disc chart $(D, |z|<1)$ around $\mathfrak{p}$ as in (1), we consider the germ $\mathfrak{f_p} = [z^{b_0}H_0(z): \ldots : z^{b_n}H_n(z)]$  where $b_0 < \ldots < b_n$ and each $H_j$ is holomorphic in $D$ and non-vanishing at $z=0$. By Lemma 4.1. in \cite{MSSX2024}, we have
    \begin{align*}
        \Lambda_n(\mathfrak{f_p}, z) &= z^{\sum_{i=0}^n b_i - \frac{n(n+1)}{2}} \cdot G_n\big(b_0, \ldots, b_n; H_0(z), \ldots, H_n(z); z\big) \\
        &\quad \cdot e_0 \wedge \ldots \wedge e_n,
    \end{align*}
    where $G_n$ is holomorphic in $D$ and non-vanishing at $z=0$. The vanishing of $\Lambda_n(\mathfrak{f_p}, z)$ at each $q_n$ and thus identically in $D$ implies that $G_n$ vanishes identically, another contradiction.
\end{itemize}

\end{proof}

Now we reach the key proposition of this subsection.

\begin{proposition}
\label{prop:ensemble}
On a compact Riemann surface \(X\), the concept of a \emph{character \(n\)-ensemble}, as outlined in Definition \ref{def:ensemble}, is equivalent to that of a \emph{non-degenerate \(n\)-tuple of one-forms}, as described in Definition \ref{def:nondeg_forms}.
\end{proposition}
\begin{proof}
By definition, a character $n$-ensemble is inherently non-degenerate when considered as an $n$-tuple of one-forms. This non-degeneracy is essential for defining character $n$-ensembles. According to the second statement of Lemma~\ref{lem:n-tuple}, any non-degenerate $n$-tuple of one-forms automatically constitutes a character $n$-ensemble. Therefore, the properties required for an $n$-tuple to be a character $n$-ensemble are exactly those that prevent degeneracy among the one-forms in the tuple.
\end{proof}

In the analysis that follows, we categorize three specified $n$-tuples of one-forms on a compact Riemann surface. We identify the nature of each tuple, highlighting whether it is degenerate, i.e.,  form a character $n$-ensemble.

\begin{example}
\label{example:n-tuples}

Let $\mathfrak{C} = \mathfrak{C}_X$ denote the infinite-dimensional real linear space of meromorphic one-forms with simple poles and purely imaginary periods on a compact Riemann surface $X$ {\rm (\cite[$\S$15]{Weyl1964}, \cite[$\S$8-1]{Springer1981} and \cite[$\S$II.4-5]{FK1992})}. Denote by $\mathfrak{C}^* = \mathfrak{C}_X \setminus \{0\}$ the non-zero elements of this space.

(1) It is possible to find two linear independent one-forms $\Omega_1,\,\Omega_2$ in $\mathfrak{C}_X^*$ such that the pair $\vv{\Omega} = (\Omega_1, \Omega_2)$ is degenerate, meaning that $\vv{\Omega}$ does not form a character $2$-ensemble. An example of such forms on the Riemann sphere are $\Omega_1 = \frac{{\rm d}z}{z}$ and $\Omega_2 = \frac{{\rm d}z}{z+1}$. By using \eqref{equ:n-tuple} and setting
    $\mathfrak{p}=1$, we can see that $(\Omega_1,\Omega_2)$ generates a line in $\mathbb{P}^2$. 
    
(2) Consider a one-form \(\Omega \in \mathfrak{C}^*\) and \(n\) nonzero real numbers \(\lambda_1, \ldots, \lambda_n\). The tuple \(\vv{\Omega} = (\lambda_1 \Omega, \ldots, \lambda_n \Omega)\) is non-degenerate if and only if all coefficients \(\lambda_j\) are mutually distinct.
{\rm We can argue as follows. Define \(y=\exp\left(\int_0^z \Omega\right)\). In a small disc chart \((D, |z| < 1)\), where \(\Omega\) is holomorphic and non-vanishing, we consider the \(n\)-th associated curve \(\Lambda_n(f)\) of the curve 
\[f(z) = \left[1: \exp\left(\int_0^z \Omega_1\right): \ldots :\exp\left(\int_0^z \Omega_n\right)\right]=\left[1:y^\lambda_1: \ldots : y^\lambda_n\right].\]
This curve is non-degenerate in \(D\), thereby confirming that $\vv{\Omega}$ is non-degenerate, as shown by the following computation}:
\begin{equation}
\label{equ:Lambda_n}
\begin{split}
\Lambda_n(f) &= f \wedge f' \wedge \ldots \wedge f^{(n)} \\
&= f \wedge \frac{\partial f}{\partial y} \wedge \ldots \wedge \frac{\partial^n f}{\partial y^n} \cdot \left(\frac{{\rm d}y}{{\rm d}z}\right)^{\frac{n(n+1)}{2}} \\
&= \begin{vmatrix}
1 & y^{\lambda_1} & \ldots & y^{\lambda_n} \\
0 & \lambda_1 y^{\lambda_1-1} & \ldots & \lambda_n y^{\lambda_n-1} \\
\vdots & \vdots & \ddots & \vdots \\
0 & \lambda_1 (\lambda_1-1)\ldots (\lambda_1-n+1) y^{\lambda_1-n} & \ldots & \lambda_n (\lambda_n-1)\ldots (\lambda_n-n+1) y^{\lambda_n-n}
\end{vmatrix} \\
&\cdot (y')^{\frac{n(n+1)}{2}} e_0 \wedge e_1 \wedge \ldots \wedge e_n \\
&= \prod_{i=1}^n \lambda_i \prod_{1 \le j < i \le n} (\lambda_i - \lambda_j) \cdot y^{\sum_{i=0}^n \lambda_i - \frac{n(n+1)}{2}} \cdot (y')^{\frac{n(n+1)}{2}}
e_0 \wedge e_1 \wedge \ldots \wedge e_n.
\end{split}
\end{equation}

(3) This example can be generalized as follows: Let $\Omega_1, \ldots, \Omega_n$ be $n$ distinct one-forms in $\mathfrak{C}^*$, and assume there is a point $\mathfrak{p}$ on $X$ where the residues of each $\Omega_j$ at $\mathfrak{p}$ are distinct nonzero numbers. Under these conditions,  $\vv{\Omega} = (\Omega_1, \ldots, \Omega_n)$ forms a non-degenerate $n$-tuple. {\rm The proof is similar to (2)}.

\end{example}

\subsection{Proof of Theorem \ref{thm:sing}} The first statement of the theorem aligns with Lemma~\ref{lem:n-tuple}(1). We proceed to demonstrate the second statement as follows:

Consider a disc chart $(D, |z| < 1)$ centered at a zero $\mathfrak{p}$ of $\vv{\Omega}$. Let $(k_1, \ldots, k_n) \in \mathbb{Z}_{>0}^n$ represent the vector of multiplicities of $\vv{\Omega} = (\Omega_1, \ldots, \Omega_n)$ at $\mathfrak{p}$. We choose a germ $\mathfrak{f_p}$ of $f_{\vec{\Omega}, \vec{\rho}}$ in $D$ as follows:
\[
\left[1 : C_1 + z^{k_1 + 1}h_1(z) : \ldots : C_n + z^{k_n + 1}h_n(z)\right],
\]
where each $h_j(z)$ is a holomorphic function in $D$ that does not vanish at $z = 0$, and each $C_j$ is a nonzero complex number.
Given that each $k_i + 1$ is an integer greater than 1, and using the calculations presented on pages 266-268 in \cite{GH:1994}, we determine that the ramification index of $f = f_0$ at $\mathfrak{p}$ is $\min(k_1, \ldots, k_n) > 0$. Consequently, $\mathfrak{p}$ is confirmed as a ramification point of $f$.

At last, we present an algorithm to identify the regular singularities of the curve 
\(f(z)=\left[1:\exp\left(\int^z \Omega_1\right): \ldots :\exp\left(\int^z \Omega_n\right)\right]\) generated by a character $n$-ensemble \(\vec{\Omega}\). The procedure is as follows:

\begin{enumerate}
    \item Cover the manifold \(X\) using a finite collection of disk charts \((D_i, |z_i| < 1)\) for \(i = 1, \dots, N\). Ensure that each pole in the set \(\mathcal{P}\) is contained within exactly one chart.
    
    \item Within these charts, and excluding the poles in \(\mathcal{P}\), solve the equation \(\Lambda_n(f) = 0\) to identify \(\mathcal{Z}\), the set of all ramification points of \(f\) not in \(\mathcal{P}\).
    
    \item For each point \(\mathfrak{p} \in \mathcal{P} \cup \mathcal{Z}\), select the appropriate chart \((D, |z| < 1)\) around it. Execute the process outlined in the proof of Lemma \ref{lem:n-tuple}(1) to determine the quasi-canonical form of \(f\). This step allows us to ascertain all \(\gamma_{\mathfrak{p}, \cdot}\) indices.
\end{enumerate}

%\subsection{Singularities of toric curves generated by an ensemble}

\section{Correspondence between character ensembles and toric solutions}

In this section, we prove Theorem \ref{thm:corr}, which establishes a correspondence between character ensembles and toric solutions with cone singularities
on a compact Riemann surface $X$.

\begin{proof} The first statement of Theorem \ref{thm:corr}  follows from Proposition \ref{prop:ensemble} and Lemma \ref{lem:n-tuple}. 

(2)  Consider a toric solution $\vv{\omega}$ with cone singularities at $\vv{D}$ on a compact Riemann surface $X$. Let $f: X \setminus \{P_1, \ldots, P_k\} \to \mathbb{CP}^n$ be an associated curve derived from this solution. This curve is characterized as a totally unramified unitary curve, and its monodromy is constrained within a maximal torus of the group $\mathrm{PSU}(n+1)$. To specifically align the monodromy of the curve within $\mathbb{T}^n$, a subgroup of the maximal torus, we select an appropriate element $\varphi \in \mathrm{PSU}(n+1)$. Applying $\varphi$, the transformed curve $\varphi \circ f: X \setminus \{P_1, \ldots, P_k\} \to \mathbb{CP}^n$ indeed has its monodromy contained within $\mathbb{T}^n$. We will proceed under the assumption that this simplification applies to the curve $f$.

Choose a sufficiently small disc chart $(D, |z|<1)$ around $\mathfrak{p} \in \{P_1, \ldots, P_k\}$. Using the argument in the proof of \cite[Theorem 3.1.]{MSSX2024}, the restriction of $f = [f_0: \ldots: f_n]$ to $D$ can be expressed as:
\begin{equation}
    f(z) = \left[z^{b_0} \phi_0(z): \ldots : z^{b_n} \phi_n(z)\right]
\end{equation}
where $b_i \in \mathbb{R}$ and $\phi_i$ is a holomorphic function that does not vanish on $D$ for all $1 \leq i \leq n$. 

We define an $n$-tuple $\vec{\Omega} = (\Omega_1, \ldots, \Omega_n)$ by:
\begin{equation}
    \Omega_k = \mathrm{d}\left(\log \frac{f_k}{f_0}\right), \quad k = 1, \ldots, n.
\end{equation}
By computation, we find:
\begin{equation}
    \Omega_k = \left((b_k - b_0) \frac{1}{z} + \frac{\phi'_k(z)}{\phi_k(z)} - \frac{\phi'_0(z)}{\phi_0(z)}\right) \mathrm{d}z
\end{equation}
Hence, each $\Omega_j$ has at most simple poles. Additionally, it can be verified that:
\begin{equation}
    2\Re\, \Omega_k = \mathrm{d} \left(\log \left|\frac{f_k}{f_0}\right|^2\right),
\end{equation}
i.e., the real part of each $\Omega_k$ is exact outside of its poles. Consequently, each $\Omega_k$ is a meromorphic one-form with simple poles and purely imaginary periods. Since $\vec{\Omega}$ is non-degenerate (as $f$ is non-degenerate), it is a character $n$-ensemble on $X$ by Proposition \ref{prop:ensemble}.

\end{proof}

By Theorem 1.5 in \cite{CWWX:2015}, the regular singularities of a toric curve to $\mathbb{P}^1$, which is derived from a character one-form, are confined to the union of the zeros and poles of the one-form. However, the following example demonstrates that a toric curve to $\mathbb{P}^n$, generated from a character \(n\)-ensemble for \(n>1\), can possess ramification points that do not coincide with the zeros or the poles of any ensemble component.

\begin{example}
Consider the toric curve defined by \( f = \left[1: z: z^2: \ldots : z^{n-1}, z^{n+1}\right] \), where \( n > 1 \). \( z=0 \) is a ramification point. The associated character \( n \)-ensemble, \( \vv{\Omega} = (\Omega_1, \ldots, \Omega_n) \), is given by:
\begin{equation*}
    \Omega_k = \frac{k}{z} \, {\rm d}z, \quad \text{for } k=1, 2, \ldots, n-1; \quad \Omega_n = \frac{n+1}{z} \, {\rm d}z
\end{equation*}
Here, \( z=0 \) is a common pole of all components \( \Omega_j \). However, applying a non-degenerate linear transformation to \( f \) results in:
$$
\tilde{f} = \left[1: 1+z: 1+z+z^2: \ldots : 1+z+z^{n-1} : 1+z+z^{n+1}\right],
$$
with the modified character \( n \)-ensemble \( \vv{\widetilde{\Omega}} = (\widetilde{\Omega}_1, \ldots, \widetilde{\Omega}_n) \) represented as:
\begin{equation*}
    \begin{split}
        &\widetilde{\Omega}_1 = \frac{{\rm d}z}{1+z}, \; \widetilde{\Omega}_2 = \frac{1+2z}{1+z+z^2} \, {\rm d}z, \ldots, \\
        &\widetilde{\Omega}_{n-1} = \frac{1+(n-1)z^{n-2}}{1+z+z^{n-1}} \, {\rm d}z, \; \widetilde{\Omega}_n = \frac{1+(n+1)z^n}{1+z+z^{n+1}} \, {\rm d}z
    \end{split}
\end{equation*}
Although \( z=0 \) remains a ramification point for \( \tilde{f} \), it is notably no longer a zero or pole of any new one-forms. 
\end{example}

\section{Two examples}
In this section, we explore the practical application of the correspondence between character ensembles and toric solutions, 
as discussed in the previous section. We conduct a thorough analysis of Examples \ref{exam:Omega} and \ref{exam:RS1}. 
These cases introduce innovative toric solutions to the ${\rm SU}(n+1)$ Toda system with cone singularities. 
Notably, these examples expand upon the recent findings detailed in \cite[Theorems 1.8-9]{LYZ2020}, extending the known boundaries of this research area. 

\subsection{Example \ref{exam:Omega}}
This subsection details Example \ref{exam:Omega}. Utilizing Example \ref{example:n-tuples} (2) and Proposition \ref{prop:ensemble}, we can see that $\Omega$ constitutes a character $n$-ensemble on $X$. This ensemble generates a family of toric curves parametrized by $(\mathbb{R}_{>0})^n$ as per Equation \eqref{equ:ensemble}, all sharing the same regular singularities, denoted by $\vv{D}$. According to Theorem \ref{thm:corr_sing}, these curves are associated with a family of toric solutions that exhibit the same cone singularities $\vv{D}$. Notations from Example \ref{example:n-tuples} (2) are used herein.

By Equation \eqref{equ:Lambda_n}, the singularities of the curve
\[f(z) = \left[1: \exp\left(\int_0^z \lambda_1 \Omega\right): \ldots: \exp\left(\int_0^z \lambda_n \Omega\right)\right] = \left[1: y^{\lambda_1}, \ldots, y^{\lambda_n}\right]\]
are limited to the set of zeros and poles of $\Omega$. Considering a point $\mathfrak{p}$ within this set, we categorize the analysis into the following two cases:

\begin{enumerate}
\item \textbf{$\mathfrak{p}$ is a zero of $\Omega$ with order $k$:} We establish that $\mathfrak{p}$ {\it is a ramification point of the curve $f$ with all ramification indices equal to} $k$. Assuming without loss of generality that $y(\mathfrak{p})=1$, we express $y(z) = 1 + z^{k+1}\phi(z)$ in a small disc chart $(D,|z|<1)$ around $\mathfrak{p}$, where $\phi(z)$ is holomorphic and non-vanishing in $D$. 

Expanding each $y^{\lambda_i}$ around $\mathfrak{p}$ into a power series, we obtain:
\begin{equation*}
    y^{\lambda_i} = 1 + \lambda_i z^{k+1} \varphi(z) + \frac{\lambda_i(\lambda_i-1)}{2} z^{2k+2} \varphi(z)^2 + \ldots, \quad i = 1, 2, \ldots, n.
\end{equation*}
Consequently, $f$ can be expressed as:
\begin{equation*}
\label{equ:infinity}
\begin{split}
    f &= (1, y^{\lambda_1}, y^{\lambda_2}, \ldots, y^{\lambda_n}) \\
    &= \left(1, z^{k+1}\varphi(z), z^{2k+2}\varphi(z)^2, \ldots\right) \cdot
    \begin{pmatrix}
        1 & 1 & \ldots & 1 \\
        0 & \lambda_1 & \ldots & \lambda_n \\
        0 & \frac{\lambda_1(\lambda_1-1)}{2} & \ldots & \frac{\lambda_n(\lambda_n-1)}{2} \\
        \vdots & \vdots & \ddots & \vdots
    \end{pmatrix}
\end{split}
\end{equation*}
Right-multiplying $f$ by an $(n+1) \times (n+1)$ invertible matrix corresponds to applying column transformations to the $\infty \times (n+1)$ infinite matrix on the right-hand side of \eqref{equ:infinity}. The first $n+1$ rows of this matrix:
\begin{equation*}
\begin{pmatrix}
    1 & 1 & \ldots & 1 \\
    0 & \lambda_1 & \ldots & \lambda_n \\
    0 & \frac{\lambda_1(\lambda_1-1)}{2} & \ldots & \frac{\lambda_n(\lambda_n-1)}{2} \\
    \vdots & \vdots & \ddots & \vdots \\
    0 & \frac{\lambda_1(\lambda_1-1)\ldots(\lambda_1-n+1)}{n!} & \ldots & \frac{\lambda_n(\lambda_n-1)\ldots(\lambda_n-n+1)}{n!}
\end{pmatrix}
\end{equation*}
forms an $(n+1) \times (n+1)$ invertible matrix, with a determinant given by:
$$\prod_{i=1}^n i! \prod_{i=1}^n \lambda_i \prod_{1 \le j < i \le n} (\lambda_i - \lambda_j) \neq 0.$$
Thus, it can be transformed into a lower triangular matrix through a finite number of column transformations, such as the Gaussian elimination process. Correspondingly, $f$ has been transformed into the quasi-canonical form:
\begin{equation*}
\label{equ:quasi-can}
    \left(1, z^{k+1}\varphi(z) + \ldots, z^{2k+2}\varphi(z)^2 + \ldots, \ldots \ldots \right)
\end{equation*}
under finitely many transforms in ${\rm PGL}(n+1,\mathbb{C})$. By the definition of ramification indices \cite[pp.\,266-268]{GH:1994} for a holomorphic curve, applying such transformations to the curve does not alter its ramification indices. Therefore, all the $n$ ramification indices of $f$ at $\mathfrak{p}$ coincide with those of the curve \eqref{equ:quasi-can}, each equal to $k$.

\item {\bf $\mathfrak{p}$ is a pole of $\Omega$:} Assume \({\rm Res}_{\mathfrak{p}}(\Omega) = a \in \mathbb{R} \setminus \{0\}\). In a small disk chart \((D, |z|<1)\) around point \(\mathfrak{p}\), the function \(y = \exp\left(\int_0^z \Omega\right)\) is represented as \(z^a \phi(z)\), where \(\phi(z)\) is holomorphic and nonvanishing within \(D\). Accordingly, we define \(f(z) = \left[1: z^{\lambda_1 a} \phi(z)^{\lambda_1}: \ldots :z^{\lambda_n a} \phi(z)^{\lambda_n}\right]\), representing a totally unramified unitary curve in \(D^*\) with a potential regular singularity at \(z=0\) by Equation \eqref{equ:Lambda_n}.

By Theorem \ref{thm:corr_sing}, \(f\) induces a solution \(\vv{\omega}\) to the \({\rm SU}(n+1)\) Toda system on \(D\), characterized by a possible cone singularity at \(z=0\). Theorem \ref{thm:corr_sing} assures that the singularities of both types at \(z=0\) are equivalent. Given the distinct real numbers \(0, \lambda_1 a, \lambda_2 a, \ldots, \lambda_n a\), we reorder them in ascending order to achieve \(\mu_0 < \mu_1 < \ldots < \mu_n\). Concurrently, the components of \(f\) are rearranged in ascending powers of \(z\) to attain the quasi-canonical form:
\begin{equation*}
	\left[z^{\mu_0} \phi(z)^{\frac{\mu_0}{a}}: z^{\mu_1} \phi(z)^{\frac{\mu_1}{a}}: \ldots : z^{\mu_n} \phi(z)^{\frac{\mu_n}{a}}\right].
\end{equation*}
The solution \(\vv{\omega}\), and consequently its cone singularity at \(z=0\), remain unaltered, as this constitutes merely a transformation within \({\rm PSU}(n+1)\). It is established that \(\gamma_{\mathfrak{p}, i} = \mu_i - \mu_{i-1} - 1\) for all \(1 \leq i \leq n\). Depending on the specific \(\lambda\) values, the point \(\mathfrak{p}\) may be classified as a branch point, a ramification point, or a totally unramified point of \(f\).

\end{enumerate}

\subsection{Example \ref{exam:Omega} provides new solutions with cone singularities on a compact Riemann surface}
The curve in this example corresponds to a novel class of solutions to the ${\rm SU}(n+1)$ Toda system with cone singularities on a compact Riemann surface. Lin-Yang-Zhong in \cite[Theorems 1.8-9]{LYZ2020} presented various sufficient conditions for the existence of such solutions on a compact Riemann surface $X$ of genus $g_X>0$. Specifically, by transforming Equation \eqref{equ:TodaSysCone}, it can be reformulated into the equivalent form:
\begin{equation*}
    \Delta_g\tilde{u}_i + \sum_{j=1}^n a_{ij}e^{\tilde{u}_j} - K_0 = 4\pi\sum_{k=1}^m\gamma_{k,i}\delta_{p_k} \quad i=1,2,\ldots,n
\end{equation*}
where $K_0$ is the Gaussian curvature function of the background conformal metric $g$ on $X$ and $\Delta_g$ its Laplacian. The function $\tilde{u}_i$ represents a global transformation of $u_i$ from Equation \eqref{equ:metric_vector}, adjusted by the background metric in each coordinate chart. The singularity coefficients $\{\gamma_{i,j}\}$ remain consistent with those described in \eqref{equ:TodaSysCone}. Lin-Yang-Zhong defined $\rho_i$ for all $1\leq i\leq n$ as follows:
\begin{equation*}
    \rho_i := 4\pi\sum_{k=1}^m \sum_{j=1}^n a^{ij}\gamma_{k,j} + \sum_{j=1}^n a^{ij} \int_X K_0
\end{equation*}
Here, $(a^{ij}) = \left(\frac{j(n+1-i)}{n+1}\right)$ represents the inverse of the Cartan matrix $(a_{ij})$ for $\mathfrak{su}(n+1)$. According to Theorem 1.9 in \cite{LYZ2020}, if $\rho_i \notin \Gamma_i$ for any $i = 1,2,\ldots n$, a solution to the Toda system is feasible. The definition of $\Gamma_i$ $(1\leq i\leq n)$ is detailed in \cite[pp.\,340-341]{LYZ2020}, noting that each includes $4\pi\mathbb{N}$.

In this example, $\Omega$ is a meromorphic one-form on a compact Riemann surface with positive genus, featuring exactly two simple poles (denoted $p_1$ and $p_2$). Assuming $0<\lambda_1<\ldots<\lambda_n$, with ${\rm Res}_{p_1}(\omega) = a \neq 0$, the corresponding $\gamma_{1,j} = (\lambda_j - \lambda_{j-1})a - 1$ for $j=1,2,\ldots,n$ (assuming $\lambda_0 = 0$ for consistency). According to the residue theorem, ${\rm Res}_{p_2}(\omega) = -a$, and as previously noted, $\gamma_{2,j} = (\lambda_{n-j+1} - \lambda_{n-j})a - 1$. Additionally, the other $m-2$ singular points $p_3,\ldots,p_m$ are identified as ramification points, with $\gamma_{i,j} \equiv k_i$ for $i = 3,\ldots,m$ and $j = 1,2,\ldots,n$, where $k_i$ denotes the order of zero of $\Omega$ at $p_i$. After calculations, we derive:
\begin{equation*}
\begin{split}
    \rho_i =& 4\pi\sum_{k=1}^m \sum_{j=1}^n a^{ij}\gamma_{k,j} + \sum_{j=1}^n a^{ij} \int_X K_0\\
    =& 4\pi\Big(\sum_{j=1}^n \frac{j(n+1-i)}{n+1} ((\lambda_j - \lambda_{j-1})a - 1) + \sum_{j=1}^n \frac{j(n+1-i)}{n+1} ((\lambda_{n-j+1} - \lambda_{n-j})a - 1) \Big) \\
    & + 4\pi \sum_{k=3}^m \sum_{j=1}^n \frac{j(n+1-i)}{n+1} k_i + \sum_{j=1}^n \frac{j(n+1-i)}{n+1} \cdot 2\pi(2-2g_X)\\
    =&4\pi(n+1-i)a\lambda_n+4\pi(k_3+\cdots+k_n-1-g_X)\cdot n(n+1-i)    
\end{split}
\end{equation*}
Choosing $a\lambda_n$ to be a positive integer, by this calculation, we can see that all $\rho_i$'s lie within $4\pi \mathbb{N} \subset \Gamma_i$ for all $1\leq i\leq n$, violating the criteria of Theorem 1.9 in \cite{LYZ2020}. Furthermore, Theorem 1.8 from Lin-Yang-Zhong also outlines sufficient conditions for the existence of solutions to the ${\rm SU}(n+1)$ Toda system, requiring all $\{\gamma_{i,j}\}$ to be non-negative integers. By choosing appropriate $\lambda_1,\ldots,\lambda_n$ such that $p_1$ and $p_2$ act as branch points, these conditions are not met. Both theorems stipulate that the genus of $X$ must be positive.

In summary, the solutions of the Toda system associated with the curves in Example \ref{exam:Omega} on a general compact Riemann surface and Example \ref{exam:RS1} on the Riemann sphere introduce new possibilities beyond the scope outlined in \cite{LYZ2020}.

%$\textbf{Example}$ 2: Toric curve in $\mathbb{CP}^2$ containing only three branch points.\\
%Consider the curve lift
%\begin{equation}
%	\hat{f}=\big(1,z^a,(z-1)^b\big)
%\end{equation}
%where $a, b \in \mathbb{R}\setminus \mathbb{Z}$. It is clear that 0, 1, and $\infty$ are the three branch points of the curve. Furthermore, after calculation:
%\begin{equation}
%	\Lambda_2(\hat{f})=ab\cdot z^{a-2}(z-1)^{b-2}\cdot\big((b-a)z+a-1\big)
%\end{equation}
%Therefore, when $a = b$, $\hat{f}$ does not have any ramified points; when $a \neq b$, the curve has a fourth singular point: the ramified point $z = \frac{1-a}{b-a}$.\\
%Note: This example only shows that in $\mathbb{CP}^2$, there exist Toric curves with exactly three singular points, all of which are branch points. Similar conclusion can be easily proven for $\mathbb{CP}^n$. This is different from the case of $\mathbb{P}^1$ (where there is no reducible conic metric with exactly three cone points, all of which have non-integer angles). However, classifying all Toric curves on $\mathbb{CP}^2$ and even $\mathbb{CP}^n$ with exactly three branch points is a more difficult problem, which the author has not yet resolved.\\ \\

\subsection{Example \ref{exam:RS1}}  We establish the existence of the necessary curves and provide an approximate enumeration of all possible singularity data at $\infty$ for these curves in this example.
\\

\noindent {\bf Existence.} The following lemma is useful for establishing existence.

\begin{lemma}
\label{lem:useful}
Let $ \alpha >0 $. For any polynomial $ P(z) $, there exists a unique polynomial 
$ Q(z) $, having the same degree as $P(z)$, such that
    \begin{equation}
    \label{equ:ode}
        \frac{{\rm d}}{{\rm d}z}\Big(z^{\alpha}Q(z)\Big)=z^{\alpha-1}P(z).
    \end{equation}
\end{lemma}

\begin{proof}
First, consider the monomial $Q(z) = \frac{z^j}{\alpha + j}$ and observe the derivative
\[
\frac{{\rm d}}{{\rm d}z}\Big(z^{\alpha}Q(z)\Big) = z^{\alpha-1}\big(\alpha Q(z) + z Q'(z)\big) = z^{\alpha-1}\cdot z^j.
\]
This calculation allows us to solve the linear ODE \eqref{equ:ode} for any polynomial $P(z)$. The uniqueness of the solution is straightforward to establish.
\end{proof}

\begin{proposition}
\label{prop:existence}
We adopt the notations from Example \ref{exam:RS1} and assume $0 = \beta_0 < \beta_1 < \ldots < \beta_n$ with the increments $\beta_{i} - \beta_{i-1} = 1+\gamma_i$ for all $1 \leq i \leq n$. Consequently, there exist $n$ polynomials $\varphi_1(z), \varphi_2(z), \ldots, \varphi_n(z)$, each non-vanishing at $z=0$, such that the family of curves
    \begin{equation*}
f_{\vv{\rho}}(z) = \left[1 : \rho_1 z^{\beta_1}\varphi_1(z) : \ldots : \rho_n z^{\beta_n}\varphi_n(z)\right], \quad \vv{\rho}=(\rho_1,\ldots,\rho_n) \in \left(\mathbb{R}_{>0}\right)^n,
    \end{equation*}
meets the requirements specified in Example \ref{exam:RS1}.
\end{proposition}

\begin{proof} We shall construct these $n$ polynomials $\varphi_1(z), \varphi_2(z), \ldots, \varphi_n(z)$ by using
the iteration argument.

{\it Step 1.} Take 
\begin{equation*}
\begin{split}
    &\psi_1(z) := \prod_{i=1}^m (z - z_i)^{\gamma_{i,1}}, \\
    &\psi_k(z) := \prod_{i=1}^m (z - z_i)^{\gamma_{i,1}} \cdot \varphi_k^{(1)}(z), \quad 2 \leq k \leq n,
\end{split}
\end{equation*}
where $\varphi_2^{(1)}(z), \ldots, \varphi_n^{(1)}(z)$ are $(n-1)$ polynomials that will be chosen later. Given that $0 < \beta_1 < \ldots < \beta_n$, Lemma \ref{lem:useful} ensures the existence of $n$ polynomials $\varphi_1(z), \ldots, \varphi_n(z)$ such that the derivative of $\hat{f}(z) := \left(1, z^{\beta_1} \varphi_1(z), z^{\beta_2} \varphi_2(z), \ldots, z^{\beta_n} \varphi_n(z)\right)$ is given by
\begin{equation}
\label{equ:d_hatf}
    \frac{{\rm d}\hat{f}(z)}{{\rm d}z} = \left(0, z^{\beta_1-1} \psi_1(z), \ldots, z^{\beta_n-1} \psi_n(z)\right) =: z^{\beta_1-1} \prod_{i=1}^m (z - z_i)^{\gamma_{i,1}} \vv{g_1}(z),
\end{equation}
where $\vv{g_1}(z) = \left(0, 1, z^{\beta_2-\beta_1} \varphi_2^{(1)}(z), \ldots, z^{\beta_n-\beta_1} \varphi_n^{(1)}(z)\right)$. Moreover, we could determine the polynomial $\varphi_1(z)$ of degree $\sum_{\ell=1}^m\,\gamma_{\ell, 1}$, which does not vanish at $z=0$.

{\it Step 2.} Differentiating $\frac{{\rm d}\hat{f}(z)}{{\rm d}z}$ again, we obtain
\begin{equation}
\label{equ:f''}
    \frac{{\rm d}^2 \hat{f}(z)}{{\rm d}z^2} =
    \frac{{\rm d}}{{\rm d}z}\left(z^{\beta_1-1} \prod_{i=1}^m (z - z_i)^{\gamma_{i,1}}\right) \cdot \vv{g_1}(z) + z^{\beta_1-1} \prod_{i=1}^m (z - z_i)^{\gamma_{i,1}} \cdot \frac{{\rm d}\vv{g_1}(z)}{{\rm d}z}.
\end{equation}
For any $(n-2)$ polynomials $\varphi_3^{(2)}(z), \ldots, \varphi_n^{(2)}(z)$ that will be chosen later, Lemma \ref{lem:useful} ensures there exist $(n-1)$ polynomials $\varphi_2^{(1)}(z), \ldots, \varphi_n^{(1)}(z)$ such that
\begin{equation}
\label{equ:dg_1}
\begin{split}
    \frac{{\rm d}\vv{g_1}}{{\rm d}z} &= \frac{\rm d}{{\rm d}z}\left(0, 1, z^{\beta_2-\beta_1} \varphi_2^{(1)}(z), \ldots, z^{\beta_n-\beta_1} \varphi_n^{(1)}(z)\right) \\
    &= z^{\beta_2-\beta_1-1} \prod_{i=1}^m (z - z_i)^{\gamma_{i,2}} \cdot \Big(0, 0, 1, z^{\beta_3-\beta_2} \varphi_3^{(2)}(z), \ldots, z^{\beta_n-\beta_2} \varphi_n^{(2)}(z)\Big) \\
    &:= z^{\beta_2-\beta_1-1} \prod_{i=1}^m (z - z_i)^{\gamma_{i,2}} \cdot \vv{g_2}(z),
\end{split}
\end{equation}
where $\vv{g_2}(z) = \left(0, 0, 1, z^{\beta_3-\beta_2} \varphi_3^{(2)}(z), \ldots, z^{\beta_n-\beta_2} \varphi_n^{(2)}(z)\right)$. Moreover, by Equations \eqref{equ:dg_1} and \eqref{equ:d_hatf}, 
we could determine these three polynomials $\varphi_2^{(1)}(z)$, $\psi_2(z)$ and $\varphi_2(z)$ such that $$\deg\, \varphi_2=\sum_{\ell=1}^m \left(\gamma_{\ell, 1}+\gamma_{\ell,2}\right)$$
and all of them do not vanish at $z=0$. 
Substituting this equation into \eqref{equ:f''} yields
\begin{equation}
    \frac{{\rm d}^2 \hat{f}(z)}{{\rm d}z^2} =
    \left(z^{\beta_1-1} \prod_{i=1}^m (z - z_i)^{\gamma_{i,1}}\right)' \cdot \vv{g_1}(z) + z^{\beta_2-2} \prod_{i=1}^m (z - z_i)^{\gamma_{i,1} + \gamma_{i,2}} \cdot \vv{g_2}(z).
\end{equation}
Repeating the above process, we continue to differentiate $\frac{{\rm d}^2 \hat{f}(z)}{{\rm d}z^2}$ and $\vv{g_2}$ and obtain 
\begin{gather*}
    \vec{g}_3 = \Big(0, 0, 0, 1, z^{\beta_4-\beta_3} \varphi_4^{(3)}(z), \ldots, z^{\beta_n-\beta_3} \varphi_n^{(3)}(z)\Big) \nonumber \\
    \vdots \nonumber \\
    \vec{g}_n \equiv \Big(0, 0, \ldots, 0, 1\Big)
\end{gather*}
In summary, for each increment in the subscript $i$ of $\vv{g_i}$, the vector-valued function $\vv{g_i}$ exhibits one additional vanishing component in its earlier entries. After carrying out this procedure $n$ times, we find all the $n$ polynomials $\varphi_1(z),\ldots, \varphi_n(z)$ such that 
\begin{equation}
\label{equ:deg}  
\deg\, \varphi_k=\sum_{j=1}^k\sum_{\ell=1}^m\, \gamma_{\ell,j},\quad 1\leq k\leq n,
\end{equation}
and all of them do not vanish at $z=0$.

{\it Step 4.}  In the final step, we show that {\it the curve 
$$f(z):=[\hat{f}(z)]= \left[1 : z^{\beta_1}\varphi_1(z) : \ldots :  z^{\beta_n}\varphi_n(z)\right]$$
satisfies the requirements specified in Example \ref{exam:RS1}}. Since  $\hat{f}(z)$ itself is of quasi-canonical form near $z=0$, it has the desired singularity information there, i.e., $\gamma_{[0],j}=\gamma_j=\beta_j-\beta_{j-1}-1$ for all $1\leq j\leq n$. 
By the induction argument, we obtain 
\begin{equation}\begin{split}
		&\hat{f}^{(k)}(z)-z^{\beta_k-k}\prod\limits_{i=1}^m(z-z_i)^{\gamma_{i,1}+\ldots+\gamma_{i,k}}\cdot \vv{g_k}(z) \\
		&=a_0^{(k)}\hat{f}(z)+ a_1^{(k)}\hat{f}'(z)+\ldots+a_{k-1}^{(k)}\hat{f}^{(k-1)}(z),\quad 1\leq k\leq n,
\end{split}\end{equation}
where all $a_i^{(j)}$ are multi-valued holomorphic functions. Then we obtain the following equation for each $1\leq k\leq n$,
\begin{equation}
\label{equ:Lambda_k}
\begin{split}
		\Lambda_k(\hat{f})&=\hat{f}\wedge\hat{f}'\wedge\ldots\wedge\hat{f}^{(k)} \\
		&=\hat{f}\wedge z^{\beta_1-1}\prod\limits_{i=1}^m(z-z_i)^{\gamma_{i,1}}\cdot \vv{g_1} \wedge \ldots \wedge z^{\beta_k-k}\prod\limits_{i=1}^m(z-z_i)^{\gamma_{i,1}+\ldots+\gamma_{i,k}}\cdot \vv{g_k} \\
  &=z^{\sum\limits_{i=1}^k\beta_i-\frac{k(k+1)}{2}}\prod\limits_{i=1}^m(z-z_i)^{k\gamma_{i,1}+(k-1)\gamma_{i,2}+\ldots+\gamma_{i,k}}
  \hat{f}\wedge \vv{g_1}\wedge\ldots\wedge \vv{g_k}\\
		&=z^{\sum\limits_{i=1}^k\beta_i-\frac{k(k+1)}{2}}\prod\limits_{i=1}^m(z-z_i)^{k\gamma_{i,1}+(k-1)\gamma_{i,2}+\ldots+\gamma_{i,k}} \cdot \\
		&\begin{pmatrix}1 \\ z^{\beta_1}\varphi_1(z) \\z^{\beta_2}\varphi_2(z) \\ \vdots \\ z^{\beta_n}\varphi_n(z)\end{pmatrix}^{\rm T} \wedge \begin{pmatrix}0 \\ 1 \\ z^{\beta_2-\beta_1}\varphi_2^{(1)}(z) \\ \vdots \\ z^{\beta_n-\beta_1}\varphi_n^{(1)}(z)\end{pmatrix}^{\rm T} \wedge \ldots \wedge \begin{pmatrix}0 \\ \vdots \\ 0 \\ 1 \\ z^{\beta_{k+1}-\beta_k}\varphi_{k+1}^{(k)}(z) \\ \vdots \\ z^{\beta_n-\beta_k}\varphi_n^{(k)}(z)\end{pmatrix}^{\rm T}.
\end{split}\end{equation}
Notably, the first term of $\hat{f} \wedge \vv{g_1} \wedge \ldots \wedge \vv{g_k}$ equals $e_0 \wedge \ldots \wedge e_k$. From this relationship, $\Lambda_n(\hat{f})$ remains non-zero outside the set $\{0, \infty, z_1, \ldots, z_m\}$, where the curve $f$ is totally unramified. Moreover, for the non-degenerate unitary curve 
$f : \mathbb{C} \setminus \{0\} \to \mathbb{P}^n$,
$\Lambda_n(\hat{f})$ vanishes at $z_1, \ldots, z_m$, which constitute all the ramification points of $f$.

For each $1 \leq \ell \leq m$, by applying a linear transformation in $\mathrm{GL}(n+1, \mathbb{C})$ to $\hat{f}$, we can achieve the quasi-canonical form of $\hat{f}$ near its ramification point $z_k$. This form is given by
\begin{eqnarray*}
\hat{f}(z) = \biggl(1, (z - z_\ell)^{\gamma_{[z_\ell], 1} + 1} g_1(z),
(z - z_\ell)^{\gamma_{[z_\ell], 1} + \gamma_{[z_\ell], 2} + 2} g_2(z),\\ \ldots, 
(z - z_\ell)^{\gamma_{[z_\ell], 1} + \ldots + \gamma_{[z_\ell], n} + n} g_n(z) \biggr),
\end{eqnarray*}
where $g_1, \ldots, g_n$ are multi-valued holomorphic functions that do not vanish at $z_\ell$. 
By straightforward computation, for each $1 \leq k \leq n$, the $k$-th associated curve $\Lambda_k(\hat{f})$ is expressed as 
\begin{equation}
\label{equ:Lambda_k_2}
	\Lambda_k(\hat{f}) = \prod_{i=1}^m (z - z_\ell)^{k \gamma_{[z_\ell], 1} + (k - 1) \gamma_{[z_\ell], 2} + \ldots + \gamma_{[z_\ell], k}} \cdot \vv{\lambda_k}(z)
\end{equation}
where $\vv{\lambda_k}(z)$ is a multi-valued holomorphic curve valued in $\Lambda^{k+1}(\mathbb{C}^{n+1})$ and does not vanish at $z_\ell$. 
By using this equation and \cite[Lemma 4.1.]{MSSX2024},   we obtain 
\begin{equation}
\label{equ:Lambda_n_2}
\Lambda_n(\hat{f}) = {\rm Const}\cdot z^{\sum_{i=1}^n\,\beta_i-\frac{n(n+1)}{2}}\prod_{\ell=1}^m (z - z_\ell)^{n \gamma_{[z_\ell], 1} + (n - 1) \gamma_{[z_\ell], 2} + \ldots + \gamma_{[z_\ell], n}}\cdot e_0\wedge\ldots\wedge e_n.
\end{equation}
Comparing this with Equation \eqref{equ:Lambda_k} yields the following system of linear equations:
\[
k \gamma_{[z_\ell], 1} + (k - 1) \gamma_{[z_\ell], 2} + \ldots + \gamma_{[z_\ell], k} = k \gamma_{\ell, 1} + (k - 1) \gamma_{\ell, 2} + \ldots + \gamma_{\ell, k}, \quad 1 \leq k \leq n.
\]
Consequently, we find that $\gamma_{[z_\ell], k} = \gamma_{\ell, k}$ for all $1 \leq \ell \leq m$ and $1 \leq k \leq n$.

\end{proof}

\begin{remark}
\label{rem:RS}
The singularity information at $\infty$ of the curve 
\[
\hat{f}(z) = \left(z^{\beta_0} \varphi_0(z), z^{\beta_1} \varphi_1(z), \ldots, z^{\beta_n} \varphi_n(z)\right) \quad \text{with} \quad \beta_0 = 1 \quad \text{and} \quad \varphi_0(z) \equiv 1,
\]
constructed in the proof is given by
\[
(\gamma_{[\infty], i})_{i=1}^n = \left(\beta_n - \beta_{n-1} - 1 + \sum_{i=1}^m \gamma_{i,n}, \ldots, \beta_2 - \beta_1 - 1 + \sum_{i=1}^m \gamma_{i,2}, \beta_1 - 1 + \sum_{i=1}^m \gamma_{i,1}\right).
\]
{\rm Indeed, substituting $z = \frac{1}{w}$ into $\hat{f}(z)$, we can obtain its local expression near $\infty$ as
\begin{equation}
\hat{f}(w) = \left(w^{-\beta_0 - \deg \varphi_0} \Phi_0(w), w^{-\beta_1 - \deg \varphi_1} \Phi_1(w), \ldots, w^{-\beta_n - \deg \varphi_n} \Phi_n(w)\right),
\end{equation}
where $\Phi_0(w), \ldots, \Phi_n(w)$ are polynomials non-vanishing at $w=0$.
Since $\beta_0 < \ldots < \beta_n$ and $\deg \varphi_0 < \ldots < \deg \varphi_n$, we obtain
the quasi-canonical form of $\hat{f}(w)$ near $w=0$ as
\[
\left(w^{-\beta_n - \deg \varphi_n} \Phi_n(w), \ldots, w^{-\beta_1 - \deg \varphi_1} \Phi_1(w), \ldots, w^{-\beta_0 - \deg \varphi_0} \Phi_0(w)\right).
\]
The statement follows from Equation \eqref{equ:deg}.}
\end{remark}

\noindent {\bf Counting.} We shall enumerate the possibilities of $(\gamma_{[\infty],i})_{i=1}^\infty$ for the curves in  Example \ref{exam:RS1}. To this end, we need the following lemma.

\begin{lemma}
\label{lem:distinct} Consider $\beta_0<\ldots<\beta_n$, $(n+1)$ polynomials $\varphi_0(z), \varphi_1, \ldots, \varphi_n(z)$ and 
	the following unitary curve on $\mathbb{C}\setminus\{0\}$
	\begin{equation*}
		\hat{f}(z)=\left(z^{\beta_0}\varphi_0(z),z^{\beta_1}\varphi_1(z),\ldots,z^{\beta_n}\varphi_n(z)\right).
	\end{equation*}
	Then by applying a suitable non-degenerate linear transformation to $\hat{f}(z)$, we can obtain
	\begin{equation}
		\hat{g}(z)= \big(z^{\beta_0}\psi_0(z), z^{\beta_1}\psi_1(z),\ldots,z^{\beta_n}\psi_n(z)\big)
	\end{equation}
	where $ \psi_0(z),\ldots,\psi_n(z) $ are polynomials such that $ \beta_0+\deg\psi_0,\beta_1+\deg\psi_1,\ldots,\beta_n+\deg\psi_n $ are mutually distinct. Moreover,  these two curves share the same 
regular singularities on the Riemann sphere $\mathbb{C}\cup \infty$.
\end{lemma}

\begin{proof}
For simplicity, we define the term “quasi-degree” for the expressions $\beta_0 + \deg \varphi_0, \beta_1 + \deg \varphi_1, \ldots, \beta_n + \deg \varphi_n$. These quasi-degrees correspond to the quasi-polynomials $z^{\beta_0}\varphi_0(z), z^{\beta_1}\varphi_1(z), \ldots, z^{\beta_n}\varphi_n(z)$, where each component can include non-integer powers of $z$.

Should any two quasi-degrees be equal, such that $\beta_i + \deg \varphi_i = \beta_j + \deg \varphi_j$ for $i < j$, we can alter $z^{\beta_i}\varphi_i(z)$ by adding a suitable multiple of $z^{\beta_j}\varphi_j(z)$. This operation transforms $z^{\beta_i}\varphi_i(z)$ into $z^{\beta_i}\tilde{\varphi}_i(z)$, where $\deg \tilde{\varphi}_i(z) < \deg \varphi_i(z)$ and $\tilde{\varphi}_i(0) \neq 0$. Importantly, the terms $z^{\beta_i}$ and $z^{\beta_j}$ remain unchanged.

Consider the scenario where $z^{\beta_m}\varphi_m(z)$ has the highest quasi-degree among all components, and it has the largest index among those with this quasi-degree. Beginning with this component, we use the aforementioned transformation to adjust other components with the same quasi-degree, reducing their degrees. This ensures that $\beta_m + \deg \varphi_m$ stands out as the unique highest quasi-degree. Subsequently, we identify the next highest quasi-degree among the remaining components and apply similar transformations. After a finite series of these steps, the quasi-degrees of all components become distinct. This results in the transformed curve maintaining its form as $\hat{g}(z) = \big(z^{\beta_0}\psi_0(z), z^{\beta_1}\psi_1(z), \ldots, z^{\beta_n}\psi_n(z)\big)$. 

Since curve $\hat{g}$ differs from curve $\hat{f}$ only by a non-degenerate linear transformation, they share the same regular singularities on the Riemann sphere.

\end{proof}

By Lemma~\ref{lem:distinct}, we assume that for each curve
\[
\hat{f}(z) = \left(z^{\beta_0} \varphi_0(z), z^{\beta_1} \varphi_1(z), \ldots, z^{\beta_n} \varphi_n(z)\right)
\]
in Example~\ref{exam:RS1}, the $(n+1)$ positive numbers $\beta_0 + \deg \varphi_0, \ldots, \beta_n + \deg \varphi_n$ are mutually distinct. Substituting $z = \frac{1}{w}$ into $\hat{f}$, we obtain its local expression near $z=\infty$, i.e., $w=0$,  as
\begin{equation}
    \hat{f}(w) = \left(w^{-\beta_0 - \deg \varphi_0} \Phi_0(w), w^{-\beta_1 - \deg \varphi_1} \Phi_1(w), \ldots, w^{-\beta_n - \deg \varphi_n} \Phi_n(w)\right)
\end{equation}
where $\Phi_0, \Phi_1, \ldots, \Phi_n$ are polynomials non-vanishing at $w=0$. Since $\beta_0 + \deg \varphi_0, \beta_1 + \deg \varphi_1, \ldots, \beta_n + \deg \varphi_n$ are pairwise distinct, we can rearrange the components in equation (4.41) according to the ascending order of the powers of $w$. This arrangement yields the quasi-canonical form of $\hat{f}(w)$ near $w=0$, which determines the sequence $(\gamma_{[\infty], i})_{i=1}^\infty$. The challenge then reduces to determining how many distinct degree vectors 
$\vv{d} = (\deg \varphi_0, \deg \varphi_1, \dots, \deg \varphi_n)$ are possible. The subsequent lemma is critical for this counting.

\begin{lemma} 
\label{lem:varphi}
$\displaystyle{\sum_{j=0}^n\,\deg\, \varphi_j=\sum\limits_{i=1}^m\Big(n\gamma_{i,1}+(n-1)\gamma_{i,2}+\ldots+\gamma_{i,n}\Big)}$.
\end{lemma}
\begin{proof}
Denote by \(A\) the sum 
$\sum_{i=1}^m \left(n\gamma_{i,1} + (n-1)\gamma_{i,2} + \ldots + \gamma_{i,n}\right)$.
Consider \(\Lambda_n(\hat{f}(z))\) as the Wronskian 
$
W_{n+1}\left(z^{\beta_0} \varphi_0(z), \ldots, z^{\beta_n} \varphi_n(z)\right)
$
of the terms \(z^{\beta_0} \varphi_0(z), \ldots, z^{\beta_n} \varphi_n(z)\). This can be expanded by columns into a linear combination of finitely many quasi-monomials of the form 
$
W_{n+1}\left(z^{\beta_0+k_0}, z^{\beta_1+k_1}, \dots, z^{\beta_n+k_n}\right)
$.
Among these, the quasi-monomial summand 
$$W_{n+1}\left(z^{\beta_0 + \deg \varphi_0}, z^{\beta_1 + \deg \varphi_1}, \dots, z^{\beta_n + \deg \varphi_n}\right)$$
in \(W_{n+1}(z^{\beta_0} \varphi_0(z), \ldots, z^{\beta_n} \varphi_n(z))\) has the highest quasi-degree, calculated as 
$\sum_{i=0}^n (\beta_i + \deg \varphi_i) - \frac{n(n+1)}{2}$.
The statement follows from equation \eqref{equ:Lambda_n_2}.
\end{proof}

By this lemma, the degree vector
$\vv{d}= (\deg\varphi_0, \deg\varphi_1, \dots, \deg\varphi_n)$ could achieve a finite number of possible values, not exceeding 
$\begin{pmatrix} 
A+n+1 \\ n\end{pmatrix}$. Therefore, we complete the counting process. 

In summary, we conclude the detailed discussion of Example \ref{exam:RS1}.

\section{Three open questions} 

We conclude this manuscript by posing the following three open questions.

\begin{ques}
{\rm Utilizing the notations from Subsection 1.3, we inquire about the characterization of the coefficient matrices 
$\Gamma:=\big(\gamma_{i,j}\big)_{\substack{1 \leq i \leq k \\ 1 \leq j \leq n}}$
associated with the regular singularities \(\vv{D}\) represented by toric solutions to the ${\rm SU}(n+1)$ Toda system on the Riemann sphere. Remarkably, Lin-Wei-Ye \cite{LWY:2012} addressed the case when \(k=2\), and Alexandre Eremenko \cite{Ere2020} resolved the case when \(n=1\) for this inquiry.}
\end{ques}

\begin{ques}
{\rm A question akin to the one previously discussed emerges for compact Riemann surfaces of positive genus. In particular, for a specified genus \(g\), we aim to characterize the \(k \times n\) matrices \(\Gamma\) that correspond to regular singularities \(\vv{D}\) manifested by toric solutions on compact Riemann surfaces of genus \(g\). It is noteworthy that Gendron-Tahar \cite{GT23} addressed this issue for the scenario when \(n=1\).}
\end{ques}

\begin{ques}
{\rm 
Consider a \(k \times n\) matrix \(\Gamma\) and a non-negative integer \(g\). The task is to determine the dimension of the moduli space of toric solutions on compact Riemann surfaces of genus \(g\), where the regular singularities \(\vv{D}\) are characterized by the coefficient matrix \(\Gamma\). Notably, Sicheng Lu and the last author explored this scenario for the case \(n=1\) in their study \cite{LX2023}.}
\end{ques}

\noindent\textbf{Acknowledgements:}
B.X. expresses sincere gratitude to Professor Guofang Wang at the University of Freiburg for introducing him to the field of Toda systems during the summer of 2018 and providing invaluable suggestions in the spring of 2024.  
Our heartfelt appreciation also goes to Professor Zhaohu Nie at the University of Utah, who kindly addressed several questions from B.X. related to Toda systems.

\bibliographystyle{plain}
\bibliography{RefBase}

\end{document}